\documentclass[10pt,a4paper]{article}

\usepackage{amsopn}

 \usepackage{graphics}
\usepackage{graphicx}
\usepackage{epsfig}

\usepackage{amssymb}
\usepackage{amsthm}
\usepackage{amsmath}
\usepackage{amsfonts}
\usepackage{lipsum}
\usepackage[left=2.5cm,right=2.5cm,bottom=3.5cm,top=3.5cm]{geometry}
\usepackage{empheq}
\usepackage{epstopdf}
\usepackage{changepage}
\usepackage{rotating}
\usepackage{adjustbox}
\usepackage{graphbox}
\usepackage{color}
\usepackage{xcolor}
\usepackage{dsfont}
\usepackage{dutchcal} 
\usepackage{float}
\usepackage{resizegather}
\usepackage{multicol}
\usepackage{booktabs} 
\usepackage{colortbl} 
\usepackage{multirow} 
\usepackage{longtable,tabu} 
\usepackage{hhline}   
\usepackage{anyfontsize} 
\usepackage{everysel}
\usepackage{psfrag}
\usepackage{titlesec}
\usepackage{bbm}
\usepackage{bm}
\usepackage{pgf,tikz}
\usetikzlibrary{arrows}

\newtheorem{theorem}{Theorem}

\newcommand{\q}{\mathbf{q}}
\newcommand{\p}{\mathbf{p}}
\newcommand{\f}{\mathbf{f}}
\newcommand{\g}{\mathbf{g}}

\renewcommand{\v}{\mathbf{v}}

\newcommand{\err}{\mathcal{r}}

\newcommand{\normal}{\mathbf{n}} 
\newcommand{\halb}{\frac{1}{2}}

\allowdisplaybreaks

\title{A new thermodynamically compatible finite volume scheme for magnetohydrodynamics}

\author{Saray Busto\thanks{Department of Applied Mathematics I, Universidade de Vigo, Campus As Lagoas, 36310 Vigo, Spain (\email{saray.busto@unitn.it}).} 
	\and 	
	Michael Dumbser\thanks{Department of Civil, Environmental and Mechanical Engineering, University of Trento, Via Mesiano 77, 38123 Trento, Italy (\email{michael.dumbser@unitn.it}).} 
}

\begin{document}

\begin{center}
	\textbf{ \Large{A new thermodynamically compatible finite volume scheme for magnetohydrodynamics} }
	
	\vspace{0.5cm}
	{S. Busto\footnote{saray.busto@uvigo.es}, M. Dumbser\footnote{ michael.dumbser@unitn.it}}
	
	\vspace{0.2cm}
	{\small
		\textit{$^{(1)}$ Department of Applied Mathematics I, Universidade de Vigo, Campus As Lagoas, 36310 Vigo, Spain}
		
		\textit{$^{(2)}$ Department of Civil, Environmental and Mechanical Engineering, University of Trento, Via Mesiano 77, 38123 Trento, Italy }
	}
\end{center}

\hrule
\vspace{0.4cm}

\begin{center}
	\textbf{Abstract}
\end{center} 

\vspace{0.1cm}
In this paper we propose a novel thermodynamically compatible finite volume scheme for the numerical solution of the equations of magnetohydrodynamics (MHD) in one and two space dimensions. As shown by Godunov in 1972, the MHD system can be written as overdetermined symmetric hyperbolic and thermodynamically compatible (SHTC) system. More precisely, the MHD equations are symmetric hyperbolic in the sense of Friedrichs and satisfy the first and second principles of thermodynamics.  
In a more recent work on SHTC systems, \cite{Rom1998}, the entropy density is a primary evolution variable, and total energy conservation can be shown to be a \textit{consequence} that is obtained after a judicious linear combination of all other evolution equations. The objective of this paper is to mimic the SHTC framework also on the discrete level by directly discretizing the \textit{entropy inequality}, instead of the total energy conservation law, while total energy conservation is obtained via an appropriate linear combination as a \textit{consequence} of the thermodynamically compatible discretization of all other evolution equations. As such, the proposed finite volume scheme satisfies a discrete cell entropy inequality \textit{by construction} and can be proven to be nonlinearly stable in the energy norm due to the discrete energy conservation. In multiple space dimensions the divergence-free condition of the magnetic field is taken into account via a new thermodynamically compatible generalized Lagrangian multiplier (GLM) divergence cleaning approach. 
The fundamental properties of the scheme proposed in this paper are mathematically rigorously proven. The new method is applied to some standard MHD benchmark problems in one and two space dimensions, obtaining good results in all cases.

\vspace{0.2cm}

\vspace{0.4cm}

\hrule

\section{Introduction}
In 1961 Godunov published a seminal paper about \textit{an interesting class of quasilinear systems}  \cite{God1961}, where the connection between symmetric hyperbolicity, in the sense of Friedrichs, \cite{FriedrichsSymm}, and thermodynamic compatibility was discovered for the first time; well before 
the paper of Friedrichs \& Lax \cite{FriedrichsLax} on the same subject. The results of \cite{God1961} 
apply to the Euler equations of compressible gasdynamics and to the shallow water equations, but not to more complex systems such as MHD or nonlinear elasticity. Later, in 1972, Godunov extended his framework also to the equations of magnetohydrodynamics (MHD), see \cite{God1972MHD}, where the divergence-free condition of the magnetic field plays a crucial role in the symmetrization process of the equations and in the proof of thermodynamic compatibility.  
Later, Godunov \& Romenski and collaborators extended the theory of symmetric hyperbolic and thermodynamic compatible (SHTC) systems to a wide class of mathematical models, including nonlinear hyperelasticity, compressible multi-phase flows and relativistic fluid and solid mechanics, see 
\cite{GodunovRomenski72,GodRom2003,Rom1998,RomenskiTwoPhase2010,Godunov2012,GRGPR}. 
In most of the above-mentioned references on SHTC systems, the \textit{entropy} density is a \textit{primary evolution variable} and the \textit{total energy} conservation law is obtained as a \textit{consequence} of all the other evolution equations, due to the privileged role the total energy plays in the underlying variational principle from which all SHTC systems can be derived. 

On the discrete level, most existing numerical methods for hyperbolic PDE discretize the total energy conservation law directly and try to achieve the discrete compatibility with the entropy inequality as a consequence, leading to the so-called entropy preserving and entropy-stable schemes, based on the seminal ideas of Tadmor \cite{Tadmor1}.  High order entropy-compatible schemes can be found in 
\cite{FjordholmMishraTadmor,CastroFjordholm,GassnerSWE,ShuEntropyMHD1,ShuEntropyMHD2,GassnerEntropyGLM,ChandrashekarKlingenberg2016,ray2016,Ray2017,Hennemann2021}, while entropy-compatible    
methods for non-conservative hyperbolic PDE were introduced, e.g., in \cite{Fjordholm2012,AbgrallBT2018}. 
Very recently, Abgrall presented a completely general framework for the construction of schemes that satisfy additional extra conservation laws, see \cite{Abgrall2018}.  

Up to now, finite volume schemes that directly discretize the entropy inequality and which obtain total energy conservation as a consequence are still very rare. Some first attempts have been recently documented in \cite{SWETurbulence} for turbulent shallow water flows and in \cite{HTCGPR} for the Euler equations and for the GPR model of continuum mechanics. 
However, to the very best of our knowledge, for the MHD system there was no scheme yet that discretizes the entropy inequality directly and that obtains total energy conservation as a mere consequence of a thermodynamically compatible discretization of the other evolution equations. It is thus the declared objective of this paper to construct such a scheme. We stress that the aim of this paper is not to introduce better or more efficient numerical schemes compared to existing methods, but to introduce a new and \textit{radically different concept}, which is the direct discretization of the entropy inequality in order to obtain the discrete total energy conservation law as a consequence. 

The rest of this paper is organized as follows. In Section \ref{sec.model}, we present the augmented MHD equations including a thermodynamically compatible GLM divergence cleaning, following the seminal ideas of Munz \textit{et al.} \cite{MunzCleaning,Dedneretal}. In Section \ref{sec.scheme.1d}, the construction of a thermodynamically compatible semi-discrete finite volume scheme is explained in great detail for the one-dimensional case. Next, in Section \ref{sec.scheme.2d}, an extension to the general multi-dimensional case is presented, together with a proof of the entropy inequality satisfied by the scheme and a proof of its nonlinear stability in the energy norm. In Section \ref{sec.results}, we present some numerical results for standard MHD benchmark problems in one and two space dimensions. The conclusions and an outlook to future work are given in Section \ref{sec.Conclusions}.

\section{The MHD equations with compatible GLM divergence cleaning}
\label{sec.model}

 We consider a formulation of the magnetohydrodynamics (MHD) equations augmented with a thermodynamically compatible generalized Lagrangian multiplier (GLM) divergence cleaning, see  \cite{MunzCleaning,Dedneretal,HyperbolicDispersion}, and including parabolic vanishing viscosity terms. Accordingly, the governing PDE system reads as follows:  
\begin{subequations}\label{eqn.MHD}
	\begin{align}
	& \frac{\partial \rho}{\partial t}+\frac{\partial (\rho v_k)}{\partial x_k} 
	 \textcolor{blue}{-\frac{\partial }{\partial x_m} \epsilon \frac{\partial \rho}{\partial x_m}} 
	= 0,\label{eqn.conti}\\[2mm]
	&\frac{\partial \rho v_i}{\partial t}+\frac{\partial \left(\rho v_i v_k + p \, \delta_{ik}  
		\textcolor{red}{\, + \, \frac{1}{2} B_m B_m \delta_{ik} - B_i B_k } \right)}{\partial x_k}  \textcolor{blue}{- \frac{\partial }{\partial x_m} \epsilon \frac{\partial \rho v_i}{\partial x_m}}  
	= 0, \label{eqn.momentum}\\[2mm]
	& 	\frac{\partial \rho S}{\partial t}+\frac{\partial \left( \rho S v_k \right)}{\partial x_k}  \textcolor{blue}{- \frac{\partial }{\partial x_m} \epsilon \frac{\partial \rho S}{\partial x_m}} 
= \textcolor{blue}{\Pi} \,   \geq 0,   \label{eqn.entropy} \\[2mm]		
	&\textcolor{red}{\frac{\partial B_{i}}{\partial t} + \frac{\partial \left( B_i v_k - v_i B_k \right) }{\partial x_k} + 
	v_i \frac{\partial B_{k}}{\partial x_k} +  c_h \frac{\partial \varphi}{\partial x_i} }  \textcolor{blue}{- \frac{\partial }{\partial x_m} \epsilon \frac{\partial B_i}{\partial x_m}} 
	=  0 ,\label{eqn.B}\\[2mm]
	& \textcolor{red}{ \frac{\partial \varphi}{\partial t} + v_k \frac{\partial \varphi }{\partial x_k} + 
	\frac{c_h}{\rho} \frac{\partial B_k}{\partial x_k} } 
    \textcolor{blue}{- \frac{\partial }{\partial x_m} \epsilon \frac{\partial \varphi}{\partial x_m}} 
    = 0, \label{eqn.phi}\\[2mm]
	& \frac{\partial \mathcal{E}}{\partial t}+\frac{\partial \left( \mathcal{E}  v_k + v_i (p \, \delta_{ik} 
		 \textcolor{red}{\, + \frac{1}{2} B_m B_m \delta_{ik} - B_i B_k }) 
		 \textcolor{red}{+ c_h \varphi B_k  } \right)}{\partial x_k} 
	   \textcolor{blue}{- \frac{\partial }{\partial x_m} \epsilon \frac{\partial \mathcal{E}}{\partial x_m}}  
	   = 0.  \label{eqn.energy} 
	\end{align}
\end{subequations}
where the black terms correspond to the Euler subsystem, the red terms are due to the presence of the magnetic field $B_i$ and the divergence cleaning scalar $\varphi$, while the vanishing viscosity terms are highlighted in blue. 
In the overdetermined system above $\mathbf{q}=\{q_i\} = (\rho, \rho v_i, \rho S, 
\textcolor{red}{B_{i}}, \textcolor{red}{\varphi})^T$ denotes the state vector, the total energy 
potential is $\mathcal{E} = \rho E =\mathcal{E}_1  +\mathcal{E}_2 \textcolor{red}{\, + \, 
\mathcal{E}_3 + \mathcal{E}_4 } $ with $\mathcal{E}_i = \rho E_i$, $\epsilon \geq 0$ is a vanishing viscosity coefficient, $c_h \geq 0$ is the divergence cleaning speed and the non-negative entropy production term due to the viscous terms is given by 
\begin{equation}
   \textcolor{blue}{\Pi = \frac{\epsilon}{T} \, \frac{\partial q_i}{\partial x_m} \,\, \partial^2_{q_i q_j} 
   	\mathcal{E} \,  \frac{\partial q_j}{\partial x_m} \geq 0}.  
\end{equation}
 Due to $\epsilon \geq 0$ and since we assume the temperature to be positive, $T > 0$, and \textcolor{black}{assuming further that the total energy potential is convex and hence its Hessian} is, at least, positive semi-definite, i.e. $\mathcal{H}_{ij}:=\partial^2_{q_i q_j} \mathcal{E} \geq 0$, the entropy production $\Pi$ is non-negative. 
Throughout this paper, we use the notations $\partial_p = \partial / \partial 
	p$ 
	and $\partial^2_{pq} = \partial^2 / (\partial p \partial q)$ for the first and second partial derivatives w.r.t. generic coordinates or quantities $p$ and $q$, which may also be vectors or components of a vector. Furthermore, 
	in the entire paper, we make use of the Einstein summation convention over repeated indices. Last but not least, in some occasions, we also use bold face symbols in order to denote vectors, e.g. $\mathbf{q}=\{q_i\}$.  
The four contributions to the total energy density, including a contribution from the cleaning scalar $\varphi$, are 
\begin{equation}
	\mathcal{E}_1 = \frac{\rho^\gamma}{\gamma-1} e^{S/c_v}, \quad 
	\mathcal{E}_2 = \halb \rho v_i v_i, 	
	\quad 
	\mathcal{E}_3 = \halb B_m B_m, 	
	\quad 
	\mathcal{E}_4 = \halb \rho \varphi^2.   	
\end{equation}
The vector of thermodynamic dual variables reads $\mathbf{p} = \mathcal{E}_\q = \left( r, 
v_i, T, 
\textcolor{red}{\beta_{i}}, \textcolor{red}{\psi} \right)^T$, with  
\begin{equation}
	r = \partial_{\rho}  \mathcal{E}, 
	\quad 
	v_i = \partial_{\rho v_i}  \mathcal{E}, 
	\quad 
	T = \partial_{\rho S}  \mathcal{E}, 
	\quad 
	\beta_{i} = \partial_{B_i}  \mathcal{E},  
	\quad 
	\psi  = \partial_{\varphi}  \mathcal{E}. 
\end{equation}
Finally, the purely hydrodynamic pressure is defined as
$
p = \rho^2 \frac{\partial E}{\partial \rho},  
$ \textcolor{black}{with $E$ the specific total energy}.  
After some calculations one can verify that \eqref{eqn.energy} is a consequence of \eqref{eqn.conti}-\eqref{eqn.phi} in the sense that 
\begin{equation}\label{eqn.sum}
\eqref{eqn.energy} = 
 r \cdot \eqref{eqn.conti} + 
 v_i \cdot \eqref{eqn.momentum} +
 T \cdot \eqref{eqn.entropy} + 
 \beta_{i} \cdot \eqref{eqn.B} + 
 \psi \cdot \eqref{eqn.phi}. 
\end{equation}

\section{Thermodynamically compatible semi-discrete finite volume scheme in one space dimension}
\label{sec.scheme.1d}
For the sake of clarity and in order to facilitate the reading, we first present the construction of our thermodynamically compatible schemes step by step in one space dimension only, using the different colors in \eqref{eqn.MHD} as guidance. We start with the inviscid Euler subsystem (black terms), then including the viscous terms (blue) and finally adding the discretization of the magnetic field and the cleaning scalar (red terms).  
Throughout this paper we use lower case subscripts, $i, \, j, \, k$, for tensor 
indices, while lower case superscripts, $\ell$, refer to the spatial discretization index. 
We denote the spatial control volumes in 1D by $\Omega^{\ell}=[x^{\ell-\halb},x^{\ell+\halb}]$
and $\Delta x = x^{\ell+\halb}-x^{\ell-\halb}$ is the uniform mesh spacing. 

\subsection{Semi-discrete Godunov formalism for the Euler subsystem}\label{sec:disc.god.form}
The Godunov form \cite{God1961} of the inviscid Euler subsystem reads  
\begin{gather} 
\label{eqn.par0} 
\left( \partial_{\p} L \right)_t + \partial_x \left( \partial_{\p} (v_1 L) \right) = 0, \\[2mm]
\q = \partial_{\p} L, \qquad \p = \partial_{\q} \mathcal{E}, \qquad \f = \partial_{\p} (v_1 L), \qquad F_G = \p \cdot \f - v_1 L.
\label{eqn.par02} 
\end{gather}
The so-called generating potential $L$, which is the Legendre transform of the total energy density, which in this section is only related to the one of the Euler subsystem, i.e. $\mathcal{E} = \mathcal{E}_{1} + \mathcal{E}_{2}$, is defined as 
\textcolor{black}{
\begin{equation}
	L = \p \cdot \q - \mathcal{E}.
\end{equation}	    
} 
A semi-discrete finite volume scheme for \eqref{eqn.par0} reads  
\begin{equation}
\label{eqn.god_sd1d}
\frac{d}{d t} \q^{\ell} = -\frac{\f^{\ell+\halb} - \f^{\ell-\halb}}{\Delta x} 
= -\frac{\left( \f^{\ell+\halb} - \f^{\ell}\right) -\left( \f^{\ell-\halb} - \f^{\ell}\right) }{\Delta x}
\end{equation}
with $\f^{\ell}=\f(\q^{\ell})$ and $\f(\q) = (\rho v_1, \rho v_i v_1 + p \delta_{i1}, 
\rho S v_1, \mathbf{0}, 0 )^T$, the fluxes of the Euler subsystem and $ F_G = v_1 (\mathcal{E}_1 + \mathcal{E}_2 + p)$ the associated total energy flux. To obtain a discrete total energy conservation law as a consequence of the discretization of \eqref{eqn.conti}-\eqref{eqn.entropy}, see \eqref{eqn.sum}, we compute the dot product of the discrete dual variables, $\p^{\ell}= \partial_\q \mathcal{E}(\q^{\ell})$, with the semi-discrete scheme \eqref{eqn.god_sd1d}:  
\begin{equation} 
\label{eqn.fluct.pt_sd1d} 
\p^\ell \cdot \frac{d}{dt} \q^\ell = \frac{d}{dt}  \mathcal{E}^\ell = - \p^\ell \cdot \frac{ \left(\f^{\ell+\halb} - \f^{\ell}\right) + \left( \f^\ell - \f^{\ell-\halb} \right) }{\Delta x} = - \frac{D_{\mathcal{E}}^{\ell+\halb,-} + D_{\mathcal{E}}^{\ell-\halb,+}}{\Delta x}.
\end{equation}
We define the energy fluctuations  
$D_{\mathcal{E}}^{\ell+\halb,-} = \p^{\ell} \cdot ( \f^{\ell+\halb} - \f^{\ell} )$, 
$D_{\mathcal{E}}^{\ell-\halb,+} = \p^{\ell} \cdot (\f^{\ell} - \f^{\ell-\halb} ) $
which must satisfy the consistency property 
\begin{equation}
\label{eqn.F.cond} 
D_{\mathcal{E}}^{\ell+\halb,-} + D_{\mathcal{E}}^{\ell+\halb,+} = 
\p^{\ell} \cdot ( \f^{\ell+\halb}  - \f^{\ell} ) + \p^{\ell+1}  \cdot (\f^{\ell+1} - \f^{\ell+\halb} ) =  F_G^{\ell+1} - F_G^{\ell}, 
\end{equation} 
in order to obtain a conservative discretization of \eqref{eqn.energy}. In \eqref{eqn.F.cond} $F_G^{\ell}$ is the discrete total energy flux $F_G$ in cell $\Omega^\ell$. Using \eqref{eqn.par02}, i.e. 
$\f^{\ell+\halb} = \partial_{\p} (v_1 L)^{\ell+\halb}$ and 
$F_G^{\ell} = \p^{\ell} \cdot \f^{\ell} - (v_1 L)^{\ell}$, one obtains 
\begin{equation*} 
- \partial_{\p} (v_1 L)^{\ell+\halb} \cdot \left( \p^{\ell+1} - \p^{\ell} \right)  + \p^{\ell+1} \cdot \f^{\ell+1} - \p^{\ell} \cdot \f^{\ell} =  
\p^{\ell+1} \cdot \f^{\ell+1} - (v_1 L)^{\ell+1} - \p^{\ell} \cdot \f^{\ell} + (v_1 L)^{\ell}.  
\end{equation*} 
Hence, the numerical flux $\f^{\ell+\halb}$ must verify the Roe-type property, 
\begin{equation}
\label{eqn.fluxcondition_sd1d}
\f^{\ell+\halb} \cdot \left( \p^{\ell+1} - \p^{\ell} \right)  =  \partial_{\p} (v_1 L)^{\ell+\halb} \cdot \left( \p^{\ell+1} - \p^{\ell} \right) = (v_1 L)^{\ell+1} - (v_1 L)^{\ell}. 
\end{equation}
Making use of the basic ideas of path conservative schemes, see \cite{Castro2006,Pares2006} for details, the path integral of the flux $\f = \partial_{\p} (v_1 L)$ satisfies the identity 
\begin{equation}
\label{eqn.pathint_sd1d}
(v_1 L)^{\ell+1} - (v_1 L)^{\ell} = \int \limits_{\p^{\ell}}^{\p^{\ell+1}} \partial_{\p} (v_1 L) \cdot d\p = \int \limits_{0}^{1} \partial_{\p} (v_1 L) \cdot \frac{\partial \boldsymbol{\psi}}{\partial s} ds,
\end{equation} 
for any path $\boldsymbol{\psi}(s), \, s\in[0,1]$ in phase space. Following \cite{SWETurbulence,HTCGPR} we choose the simple straight line segment path in $\p$ variables, i.e.  
\begin{equation} 
\label{eqn.path1_sd1d} 
\boldsymbol{\psi}(s) = \p^{\ell} + s \left( \p^{\ell+1} - \p^{\ell} \right), \qquad \frac{\partial \boldsymbol{\psi}}{\partial s} = \p^{\ell+1} - \p^{\ell}, \qquad 0 \leq s \leq 1.
\end{equation}  
Inserting the segment path \eqref{eqn.path1_sd1d} into \eqref{eqn.pathint_sd1d} yields 
\begin{equation}
\label{eqn.pathint.seg1_sd1d}
(v_1 L)^{\ell+1} - (v_1 L)^{\ell}  = \left( \int \limits_{0}^{1} \f(\boldsymbol{\psi}(s)) 
ds \right) \cdot \left( \p^{\ell+1} - \p^{\ell} \right). 
\end{equation} 
As a result, the thermodynamically compatible numerical flux for the inviscid Euler subsystem that guarantees \eqref{eqn.fluxcondition_sd1d} by construction reads  
\begin{equation}
\label{eqn.p.scheme_sd1d}
\f^{\ell+\halb}_\p = \int \limits_{0}^{1} \f(\boldsymbol{\psi}(s)) ds = 
\left( f^{\ell+\halb}_\rho, \f^{\ell+\halb}_{\rho \v}, f^{\ell+\halb}_{\rho S}, \mathbf{0}, 0 \right)^T,  
\end{equation} 
and can be approximated using a sufficiently accurate numerical quadrature rule, see e.g. \cite{OsherNC}. In particular, throughout this paper, we employ a standard Gauss-Legendre quadrature rule with $n_{GP}=3$ points. 
For a quantitative study of the influence of the quadrature rule on total energy conservation for the Euler subsystem the reader is referred to \cite{HTCGPR}. 
		
\subsection{Compatible scheme with dissipation terms}

In order to obtain a dissipative thermodynamically compatible scheme for \eqref{eqn.MHD}, we need to add a compatible numerical dissipation to the inviscid flux \eqref{eqn.p.scheme_sd1d} derived in the previous section. 
For this purpose, we augment \eqref{eqn.god_sd1d} with an additional dissipative flux and a corresponding entropy production term as follows:
\begin{equation}
\label{eqn.goddis_sd1d}
\frac{d}{dt} \q^{\ell} + \frac{\f^{\ell+\halb} - \f^{\ell-\halb}}{\Delta x} = \textcolor{blue}{\frac{\g^{\ell+\halb} - \g^{\ell-\halb}}{\Delta x} + \mathbf{P}^{\ell}}.
\end{equation}
The dissipative part of the numerical flux is taken of the form 
\begin{equation}
\g^{\ell+\halb} = \epsilon^{\ell+\halb} \frac{\Delta \q^{\ell+\halb}}{\Delta x}, \qquad \Delta \q^{\ell+\halb}= \q^{\ell+1}-\q^{\ell}, 
\end{equation}
where the scalar numerical dissipation coefficient can either be simply chosen as a constant, i.e.  $\epsilon^{\ell+\halb}=\epsilon$, or we take it of the form 
\begin{equation}
\epsilon^{\ell+\halb} = \halb \left( 1 - \phi^{\ell+\halb} \right) \Delta x \, s_{\max}^{\ell+\halb} \geq 0, 
\label{eqn.viscosity} 
\end{equation}
where $s_{\max}^{\ell+\halb}$ is the maximum signal speed at the cell interface (Rusanov flux) and  $\phi^{\ell+\halb}$ is a flux limiter that allows to reduce the numerical dissipation in smooth regions. In this paper we use the minbee flux limiter that reads 
\begin{equation}
\phi^{\ell+\halb} = \min \left( \phi^{\ell+\halb}_{-}, \phi^{\ell+\halb}_+  \right), \quad \textnormal{with} \quad 
\phi^{\ell+\halb}_{\pm} = \max\left( 0, \min \left(1, h^{\ell+\halb}_{\pm} \right) \right),
\label{eqn.fluxlimiter} 
\end{equation}
where the ratios of density slopes, similar to the SLIC scheme \cite{toro-book}, are 
\begin{equation}
h^{\ell+\halb}_{-} = \frac{\rho^{\ell} - \rho^{\ell-1}}{ \rho^{\ell+1} - \rho^{\ell} }, \qquad \textnormal{and} \qquad 
h^{\ell+\halb}_{+} = \frac{\rho^{\ell+2} - \rho^{\ell+1}}{ \rho^{\ell+1} - \rho^{\ell} }.
\end{equation}

Calculating the dot product of $\p^{\ell}$ with \eqref{eqn.goddis_sd1d} leads to 
\begin{equation}
\label{eqn.pgoddis_sd1d}
\frac{d\mathcal{E}^{\ell}}{dt} + \frac{1}{\Delta x} \left(F_G^{\ell+\halb}-F_G^{\ell-\halb}\right)
= \frac{1}{\Delta x} \p^{\ell}\cdot \left(\g^{\ell+\halb} - \g^{\ell-\halb}\right) 
+\p^{\ell}\cdot \mathbf{P}^{\ell},
\end{equation}
\textcolor{black}{
with the numerical energy flux that can be computed from the total energy fluctuation at the interface as $F_G^{\ell+\halb} = D_{\mathcal{E}}^{\ell+\halb,-}+F_G^{\ell}$.}  
The thermodynamic compatibility of the left hand side of \eqref{eqn.pgoddis_sd1d} has already been studied in the previous section, hence, in the following, we can focus on the right hand side of \eqref{eqn.pgoddis_sd1d} obtaining 
\begin{eqnarray} 
\label{eqn.E.diss.2} 
\p^{\ell} \cdot \mathbf{P}^{\ell} + \p^{\ell} \cdot \frac{\g^{\ell+\halb} - \g^{\ell-\halb}}{\Delta x} && \nonumber \\  
=\p^{\ell} \cdot \mathbf{P}^{\ell} + \frac{1}{\Delta x} \left( \halb \p^{\ell} \cdot \g^{\ell+\halb} + \halb \p^{\ell+1} \cdot \g^{\ell+\halb} + 
\halb \p^{\ell} \cdot \g^{\ell+\halb} - \halb \p^{\ell+1} \cdot \g^{\ell+\halb} \right)  && \nonumber \\ 
-\frac{1}{\Delta x} \left( \halb \p^{\ell} \cdot \g^{\ell-\halb} + \halb \p^{\ell-1} \cdot \g^{\ell-\halb} + \halb \p^{\ell} \cdot \g^{\ell-\halb} - \halb \p^{\ell-1} \cdot \g^{\ell-\halb} \right)     && \nonumber \\ 
=\p^{\ell} \cdot \mathbf{P}^{\ell}  +  \halb \frac{\p^{\ell+1} + \p^{\ell}}{\Delta x}  \cdot \epsilon^{\ell+\halb} \frac{\Delta \q^{\ell+\halb}}{\Delta x} - \halb \frac{ \p^{\ell} + \p^{\ell-1}}{\Delta x} \cdot \epsilon^{\ell-\halb} \frac{\Delta \q^{\ell-\halb}}{\Delta x}   && \nonumber \\   
- \halb \frac{\p^{\ell+1} - \p^{\ell}}{\Delta x} \cdot \epsilon^{\ell+\halb} \frac{\Delta \q^{\ell+\halb}}{\Delta x}  
- \halb \frac{\p^{\ell} - \p^{\ell-1}}{\Delta x} \cdot \epsilon^{\ell-\halb} \frac{\Delta \q^{\ell-\halb}}{\Delta x}. 
\end{eqnarray}
Thanks to the identity 
\begin{equation}
\label{eqn.pathintegraldis_sd1d}
\int \limits_{\q^{\ell}}^{\q^{\ell+1}} \p \, \cdot \, d \q  =  \int 
\limits_{\q^{\ell}}^{\q^{\ell+1}} \partial_\q \mathcal{E} \cdot d \q = \mathcal{E}^{\ell+1} - 
\mathcal{E}^{\ell} = \Delta \mathcal{E}^{\ell+\halb},
\end{equation}
we can interpret the term 
$\halb ( \p^{\ell+1} + \p^{\ell} ) \cdot \Delta \q^{\ell+\halb}$
as an \textit{approximation} of the difference of the total energy density $\Delta \mathcal{E}^{\ell+\halb}$, where the above path integral has been evaluated via the trapezoidal rule. 
As a consequence of \eqref{eqn.E.diss.2} and \eqref{eqn.pathintegraldis_sd1d}, the energy flux including convective and diffusive terms is 
\begin{equation}
F^{\ell+\halb}_d = F_G^{\ell+\halb} - \halb ( \p^{\ell+1} + \p^{\ell} ) \cdot \epsilon^{\ell+\halb} \frac{\Delta \q^{\ell+\halb}}{\Delta x} 
\approx   F_G^{\ell+\halb} -  \epsilon^{\ell+\halb} \frac{\Delta \mathcal{E}^{\ell+\halb}}{\Delta x}.
\end{equation} 
\noindent In order to obtain a quadratic form of which we can control the sign, we now need to rewrite jumps in $\p$ variables in terms of jumps in $\q$ variables. For that purpose, we need a Roe-type matrix $\partial^2_{\q\q} \tilde{\mathcal{E}}^{\ell+\halb}$ that verifies the Roe property 
\begin{equation}
\label{eqn.roeprop2}
\partial^2_{\q\q} \tilde{\mathcal{E}}^{\ell+\halb} \cdot (\q^{\ell+1}-\q^{\ell} ) = \p^{\ell+1} - \p^{\ell}.
\end{equation}
\noindent The following segment path $\tilde{\boldsymbol{\psi}}$ in terms of the $\q$ variables    
\begin{equation} 
\label{eqn.path2_sd1d} 
\tilde{\boldsymbol{\psi}}(s) =  \q^{\ell} + s \left( \q^{\ell+1} - \q^{\ell} \right), \qquad 0 \leq s \leq 1,  
\end{equation}  
allows us to obtain the Roe matrix we are looking for: 
\begin{equation}
\tilde{\boldsymbol{\mathcal{H}}}^{\ell+\halb} = \partial^2_{\q\q}\tilde{\mathcal{E}}^{\ell+\halb} = \int \limits_0^1 \partial^2_{\q\q} \mathcal{E}\left(\tilde{\boldsymbol{\psi}}(s)\right) ds =: 
\left( \partial^2_{\p\p} \tilde{L}^{\ell+\halb} \right)^{-1}.
\label{eqn.EqqLqqRoe}
\end{equation} 
\textcolor{black}{Like the entropy compatible flux \eqref{eqn.p.scheme_sd1d} the Roe matrix \eqref{eqn.EqqLqqRoe} above is computed via numerical quadrature. Throughout this paper we employ a three-point Gauss-Legendre quadrature rule for its computation.} The Roe matrix \eqref{eqn.EqqLqqRoe} satisfies \eqref{eqn.roeprop2} and allows to rewrite \eqref{eqn.pgoddis_sd1d}, after substitution of \eqref{eqn.E.diss.2}, as 
\begin{eqnarray} 
\label{eqn.E.diss.3} 
\frac{d}{dt} \mathcal{E}^{\ell} + \frac{F_d^{\ell+\halb} - F_d^{\ell-\halb}}{\Delta x} = \p^{\ell} \cdot \mathbf{P}^{\ell} &&  \\ 
- \halb \epsilon^{\ell+\halb} \frac{\q^{\ell+1} - \q^{\ell}}{\Delta x} \cdot \tilde{\boldsymbol{\mathcal{H}}}^{\ell+\halb} \frac{\q^{\ell+1}-\q^{\ell}}{\Delta x}  
- \halb \epsilon^{\ell-\halb} \frac{\q^{\ell} - \q^{\ell-1}}{\Delta x} \cdot \tilde{\boldsymbol{\mathcal{H}}}^{\ell-\halb} \frac{\q^{\ell}-\q^{\ell-1}}{\Delta x}. \nonumber
\end{eqnarray}
The only equation in \eqref{eqn.MHD} that allows for a non-negative production term on the right hand side is the entropy inequality. Therefore, we define the production term $\mathbf{P}^{\ell} = (0,\mathbf{0},\Pi^{\ell},\mathbf{0},0)^T$ as follows, in order to cancel all contributions on the right hand side of the equality sign,   
\begin{equation} 
\p^{\ell} \cdot \mathbf{P}^{\ell} = T^{\ell} \Pi^{\ell} = 
\halb \epsilon^{\ell+\halb} \frac{\Delta \q^{\ell+\halb}}{\Delta x} \cdot \tilde{\boldsymbol{\mathcal{H}}}^{\ell+\halb} \frac{\Delta \q^{\ell+\halb}}{\Delta x} +  
\halb \epsilon^{\ell-\halb} \frac{\Delta \q^{\ell-\halb}}{\Delta x} \cdot \tilde{\boldsymbol{\mathcal{H}}}^{\ell-\halb} \frac{\Delta \q^{\ell-\halb}}{\Delta x},
\label{eqn.production.condition} 
\end{equation}
which yields to the sought semi-discrete total energy conservation law 
\begin{equation} 
\label{eqn.E.diss} 
\frac{d}{dt} \mathcal{E}^{\ell} + \frac{F_d^{\ell+\halb} - F_d^{\ell-\halb}}{\Delta x} = 0.    
\end{equation}
The final thermodynamically compatible Rusanov flux, which includes convective and diffusive terms, is therefore given by 
\begin{eqnarray} 
\f^{\ell+\halb}_{\p,d} &=& 
\int \limits_{0}^{1} \f(\boldsymbol{\psi}(s)) ds  
- \frac{\epsilon^{\ell+\halb}}{\Delta x} \left( \q^{\ell+1} - \q^{\ell} \right).  
\label{eqn.p.scheme.diss} 
\end{eqnarray}

\subsection{Compatible discretization of the magnetic field} 
To derive a compatible the discretization of the terms related to the magnetic field, we gather them in two different relations mimicking what can be observed at the continuous level, i.e. 
\begin{gather}
	-v_{i}\frac{\partial B_{i}B_{1}}{\partial x} - B_i \frac{\partial v_{i}B_{1}}{\partial x} + B_i v_{i} \frac{\partial B_{1}}{\partial x}  =  -  \frac{\partial v_{i} B_i B_{1}}{\partial x}, \label{eqn.Bcompatible1.cont}\\
	\halb v_i \frac{\partial B_mB_m\delta_{i1}   }{\partial x}  + B_i \frac{\partial B_i v_1   }{\partial x} 
	= \frac{\partial  \halb B_{m} B_{m} v_1  }{\partial x} + \halb \frac{\partial v_{i} B_mB_m\delta_{i1} }{\partial x}.
	\label{eqn.Bcompatible2.cont}
\end{gather}

First, for the compatible discretization related to \eqref{eqn.Bcompatible1.cont}, we define the auxiliary notation $R_{ik}:=-B_iB_k$ for the discrete magnetic stress tensor so for compatibility with energy conservation \textcolor{black}{at the discrete level we must require 
that the sum of the total energy fluctuations is equal to the difference of the corresponding total energy fluxes, i.e.} 
\begin{eqnarray*}
	v_{i}^{\ell}\left(R_{i1}^{\ell+\halb}-R_{i1}^{\ell} \right) + v_{i}^{\ell+1}\left(R_{i1}^{\ell+1}-R_{i1}^{\ell+\halb} \right) 
	- B_i^{\ell} \left(  \left( v_{i}B_{1}\right)^{\ell+\halb} -  \left( v_{i}B_{1}\right)^{\ell} \right) \notag\\ + B_i^{\ell+1} \left(  \left( v_{i}B_{1}\right)^{\ell+1} -  \left( v_{i}B_{1}\right)^{\ell+\halb} \right)
	+\halb B_i^{\ell} v_{i}^{\ell + \halb}\left(B_1^{\ell+1} -B_1^{\ell}\right)   \notag\\
	+\halb B_i^{\ell+1} v_{i}^{\ell + \halb}\left(B_1^{\ell+1} -B_1^{\ell}\right)=   \left( v_{i}R_{i1}\right)^{\ell+1} - \left( v_{i}R_{i1}\right)^{\ell}. 
\end{eqnarray*}
Assuming
\begin{equation*}
	\left(v_iB_1 \right)^{\ell+\halb} :=  \halb \left( \left(v_i B_1 \right)^{\ell}+ \left(v_i B_1 \right)^{\ell+1} \right),\qquad v_i^{\ell+\halb} := \halb\left(v_i^{\ell} + v_i^{\ell+1}  \right)\end{equation*}
after some algebra, we get
\begin{equation*}
	-R_{i1}^{\ell+\halb} \left(v_{i}^{\ell+1}- v_{i}^{\ell}\right) 
	- \halb \left( B_i^{\ell}+ B_i^{\ell+1}\right) \halb \left( B_{1}^{\ell} + B_{1}^{\ell+1} \right) \left(v_{i}^{\ell+1}- v_{i}^{\ell}\right)   =0, 
\end{equation*}
thus the compatible discretization for the magnetic stress tensor $R_{i1}^{\ell+\halb}$ reads
\begin{equation}
	R_{i1}^{\ell+\halb} = - \halb \left( B_i^{\ell}+ B_i^{\ell+1}\right) \halb \left( B_{1}^{\ell} + B_{1}^{\ell+1} \right).
\end{equation}

Introducing $\mu := \halb B_m B_m$ for the magnetic pressure and imposing \eqref{eqn.Bcompatible2.cont} at the discrete level yields
\begin{eqnarray*}
	v_i^{\ell} \left(\mu^{\ell+\halb} - \mu^{\ell} \right) \delta_{i1} + v_i^{\ell+1} \left(\mu^{\ell+1} - \mu^{\ell+\halb} \right) \delta_{i1} 
	+ B_i^{\ell} \left( \left( B_iv_1\right)^{\ell+\halb} -  \left( B_iv_1\right)^{\ell} \right) \notag\\
	+ B_i^{\ell+1} \left( \left( B_i v_1\right)^{\ell+1} -  \left( B_i v_1\right)^{\ell+\halb} \right) 
	= 2 \left(\left( v_1 \mu\right)^{\ell+1} - \left( v_1 \mu\right)^{\ell} \right).
\end{eqnarray*}
Taking $\left(B_iv_1\right)^{\ell+\halb}:= \halb \! \left(B_i^{\ell} + B_i^{\ell+1}\right) \halb \! \left(v_1^{\ell} + v_1^{\ell+1}\right)$, adding and subtracting $\frac{1}{4} B_m^{\ell} B_m^{\ell} v_1$, $\frac{1}{4} B_m^{\ell+1} B_m^{\ell+1} v_1^{\ell}$ and simplifying terms provides the sought compatible discretization for the discrete magnetic pressure flux:
\begin{equation}
	\mu^{\ell+\halb} = \halb\left( B_mB_m \right)^{\ell+\halb}= \frac{1}{2}\left( \halb B_m^{\ell+1}B_m^{\ell+1} + \halb B_m^{\ell}B_m^{\ell} \right).
\end{equation}	

\subsection{Compatible discretization of the GLM divergence cleaning}
For the compatible discretization of the GLM divergence cleaning, we first note that multiplication of \eqref{eqn.conti} by the dual variable $\partial_{\rho}\mathcal{E}$ yields a new term in the flux not accounted for in the Godunov formalism of the Euler subsystem of Section~\ref{sec:disc.god.form},
\begin{equation*}
	 \frac{\partial \mathcal{E}_4}{\partial \rho} \frac{\partial\rho v_1}{\partial x} = 
	 E_4 \frac{\partial \rho v_1}{\partial x} = 
	 \halb\varphi^2 \frac{\partial \rho v_1}{\partial x} .
\end{equation*}
Considering the former term, multiplying $v_1\frac{\partial \varphi}{\partial x}$ from \eqref{eqn.phi} by the dual variable $\psi$ and imposing the compatibility condition with the term $\frac{\partial \mathcal{E}_4 v_1}{\partial x}$ of the energy equation \eqref{eqn.energy},  we have
\begin{eqnarray*}
	E_4^{\ell} \left(\left(\rho v_1 \right)^{\ell+\halb}- \left(\rho v_1 \right)^{\ell} \right) 
	+ E_4^{\ell+1} \left(\left(\rho v_1 \right)^{\ell+1}- \left(\rho v_1 \right)^{\ell+\halb} \right)
	+\halb \tilde{v}_1^{\ell+ \halb} \psi^{\ell} \left(\varphi^{\ell+1}-\varphi^{\ell}\right) \\
	+\halb \tilde{v}_1^{\ell+ \halb} \psi^{\ell+1} \left(\varphi^{\ell+1}-\varphi^{\ell}\right)
	= \left(\rho v_1 E_4 \right)^{\ell+1} -\left(\rho v_1 E_4 \right)^{\ell}.
\end{eqnarray*}
Hence, making use of the inviscid mass flux $\left(\rho v_1 \right)^{\ell+\halb} = f_{\rho}^{\ell+\halb}$, which is known from the discrete Godunov formalism presented in Section \ref{sec:disc.god.form}, the compatible discretization of the advection speed in the cleaning equation is 
\begin{equation}
	\tilde{v}_1^{\ell+\halb} = \frac{\left( \rho v_1\right)^{\ell+\halb}\left(E_4^{\ell+1} - E_4^{\ell}\right) }{\halb \left(\left( \rho\varphi\right)^{\ell}+\left( \rho\varphi\right)^{\ell+1}  \right) \left(\varphi^{\ell+1}-\varphi^{\ell} \right) }.
\end{equation}
\textcolor{black}{In the case the denominator tends to zero, we compute the advection speed by simply taking the arithmetic average of the velocities in order to avoid division by zero.}  
Finally, we gather all terms related to the cleaning variable and $c_h$ and impose compatibility with the energy conservation law:
\begin{eqnarray*}	
	B_1^{\ell} c_{h} \left(\varphi^{\ell+\halb}-\varphi^{\ell}\right) 
	+ B_1^{\ell+1} c_{h} \left(\varphi^{\ell+1}-\varphi^{\ell+\halb}\right) 
	+ \rho^{\ell} \varphi^{\ell} \frac{c_{h}}{\rho^{\ell+\halb,-}}\halb\left(  B_1^{\ell+1}- B_1^{\ell} \right) \\
	+ \rho^{\ell+1} \varphi^{\ell+1} \frac{c_{h}}{\rho^{\ell+\halb,+}}\halb\left(  B_1^{\ell+1}- B_1^{\ell} \right)
	= c_h \left( \left( \varphi B_1\right)^{\ell+1} - \left( \varphi B_1\right)^{\ell}\right).
\end{eqnarray*}
Taking $\rho^{\ell+\halb,\pm}:=\halb \left( \rho^{\ell} + \rho^{\ell+1} \right)$, after some algebra, we obtain the weighted average 
\begin{equation}
\varphi^{\ell+\halb} = \frac{\rho^{\ell} \varphi^{\ell} + \rho^{\ell+1}\varphi^{\ell+1}}{\rho^{\ell} + \rho^{\ell+1}},    
\end{equation}
for the numerical flux of the cleaning scalar, which completes the derivation of a compatible discretization of \eqref{eqn.MHD} in one space dimension.

\section{Thermodynamically compatible semi-discrete finite volume scheme in two space dimensions}
\label{sec.scheme.2d}
The former procedure can be followed along a generic normal direction in space, which allows to obtain a thermodynamically compatible semi-discrete finite volume scheme also in multiple space dimensions. Considering a spatial control volume $\Omega_{\ell}$ with circumcenter $x^{\ell}$ and its set of neighbours $N_{\ell}$, one of those neighbours $\Omega^{\err}$ and the shared face $\partial \Omega^{\ell\err}$ with outward unit normal vector $\mathbf{n}^{\ell\err}=\left(n_1^{\ell\err},n_2^{\ell\err}\right)^{T}$ pointing from $x^{\ell}$ to $x^{\err}$, the final semi-discrete finite volume HTC scheme in multiple space dimensions reads
\begin{subequations}\label{eqn.schemeMHD2D}
	\begin{align}
		 \frac{\partial \rho^{\ell}}{\partial t} = & \frac{1}{\left|\Omega^{\ell}\right|} \sum_{\err\in N_{\ell}} \left|\partial\Omega^{\ell\err}\right|\left( - D_{\rho}^{\ell \err,-}
		\textcolor{blue}{+  g_{\rho,\,\normal}^{\ell\err}}\right) , \label{eqn.schemeMHD2D_mass}\\
		\frac{\partial (\rho v_{i}^{\ell})}{\partial t} = & \frac{1}{\left|\Omega^{\ell}\right|} 
		\sum_{\err\in N_{\ell}} 
		\left|\partial\Omega^{\ell\err}\right|\left( - D_{\rho v_{i}}^{\ell \err,-}
		\textcolor{red}{-  \mu^{\ell\err,\,-} n_i^{\ell\err} }  \textcolor{red}{- R_{ik}^{\ell\err,\,-}n_k^{\ell\err}}
		\textcolor{blue}{+  g_{\rho v_{i},\,\normal}^{\ell\err}} \right),\label{eqn.schemeMHD2D_momentum} \\
		\frac{\partial (\rho S^{\ell})}{\partial t} = & \frac{1}{\left|\Omega^{\ell}\right|} \sum_{\err\in N_{\ell}} 	\left|\partial\Omega^{\ell\err}\right| \left( -
	 D_{\rho S}^{\ell \err,-} \textcolor{blue}{+ g_{\rho S,\,\normal}^{\ell\err}}
		\textcolor{blue}{+ \Pi^{\ell\err,-}_{\normal}}\right) , \label{eqn.schemeMHD2D_entropy}\\
		\frac{\partial B_i}{\partial t} = &\frac{1}{\left|\Omega^{\ell}\right|} 
		\sum_{\err\in N_{\ell}}
		\left|\partial\Omega^{\ell\err}\right|\left(  \textcolor{red}{- \mathcal{D}_{B_i}^{\ell \err,-}} \textcolor{red}{+ c_{h} 
			\left( \varphi^{\ell\err} - \varphi^{\ell}\right)  n_i^{\ell\err}}
		\textcolor{blue}{+ g_{B_i,\,\normal}^{\ell\err}}\right), \label{eqn.schemeMHD2D_B}\\
		\frac{\partial \varphi}{\partial t} = &\frac{1}{\left|\Omega^{\ell}\right|} 
		\sum_{\err\in N_{\ell}}
		\left|\partial\Omega^{\ell\err}\right|\left(  \textcolor{red}{- \halb \tilde{u}^{\ell\err} \left(\varphi^{\err}-\varphi^{\ell}\right) - \frac{c_h}{\rho^{\ell\err}} \halb\left(B_{k}^{\err}-B_{k}^{\ell}\right)n_k^{\ell\err}}  
		\textcolor{blue}{+ g_{\varphi,\,\normal}^{\ell\err}}\right), \label{eqn.schemeMHD2Dphi}
	\end{align}
\end{subequations}
where
\begin{gather}
	D_{\q}^{\ell \err,-} = \left(f_{\q,\,k}^{\ell\err}-f_{\q,\,k}^{\ell}\right)n_k^{\ell\err},  \qquad
	f(\q) = (\rho \v, \left( \rho \v + p \mathbf{I}\right)  \otimes \v, 	\rho S \v, \mathbf{0} , \mathbf{0} )^T,
	\label{eqn.fluctuations_2D}\\
	\mathcal{D}_{B_i}^{\ell \err,-} = \left(\left( B_i v_k\right) ^{\ell\err} - \left(  B_i v_k\right) ^{\ell} \right)n_k^{\ell\err}
	- \left(\left( v_iB_k\right) ^{\ell\err} -  \left( v_i B_k\right)^{\ell}  \right)n_k^{\ell\err} +\halb \tilde{v}_i^{\ell\err} \left( B_k^{\err}-B_k^{\ell}\right)  n_k^{\ell\err}, \\
	g_{\q,\,\normal}^{\ell\err} = 
	\epsilon^{\ell\err}\frac{\q^{\err}-\q^{\ell}}{\delta^{\ell\err}}= 
	\epsilon^{\ell\err}\frac{\Delta
		\q^{\ell\err}}{\delta^{\ell\err}}, \quad \delta^{\ell\err} = \left\| \mathbf{x}^\err - \mathbf{x}^\ell \right\| = \Delta x\, n_{1}^{\err\ell}+\Delta y \, n_{2}^{\ell\err}, 
	\label{eqn.diffusion_2D}\\
	\Pi^{\ell\err,-}_{\normal} = \frac{1}{2}\epsilon^{\ell\err} \frac{\Delta \q^{\ell\err}}{T^{\ell}} \cdot  \partial^2_{\q \q} \mathcal{E}^{\ell\err} \frac{\Delta \q^{\ell\err}}{\delta^{\ell\err}}, \qquad 
	T^{\ell} = \frac{\left( \rho^{\ell}\right)^{\gamma-1}}{\left( \gamma-1\right) c_{v}} e^{\frac{ S^{\ell}}{c_{v}}} ,\label{eqn.production_2D}\\
	R_{ik}^{\ell\err,-} = R_{ik}^{\ell\err} -R_{ik}^{\ell} ,
	\quad R_{ik}^{\ell\err} = -\halb\left(B_{ik}^{\ell}+B_{ik}^{\err} \right) \halb \left( B_{ik}^{\ell}+B_{ik}^{\err} \right)  , 
	\label{eqn.R_2D} \\
	\mu^{\ell\err,-}  = \mu^{\ell\err}-\mu^{\ell}, \quad \mu^{\ell\err}= \halb \left( \halb B_{m}^{\err} B_{m}^{\err}  +    \halb B_{m}^{\ell} B_{m}^{\ell}\right)
	,\quad \tilde{v}_i^{\ell\err}  = \halb \left(\tilde{v}_i^{\ell}+\tilde{v}_i^{\err}\right), \label{eqn.B2_2D}\\
	\left( v_iB_k\right)^{\ell\err} = \halb\left(\left(v_iB_k\right)^{\ell}+\left(v_iB_k\right)^{\err}\right),\quad
	\left( B_i v_k \right)^{\ell\err} =  \halb\left(B_i^{\ell}+B_{i}^{\err}\right) \halb \left(v_k^{\ell}+v_k^{\err}\right), \\
	\tilde{u}^{\ell\err}  = \frac{f_{\rho,\,k}^{\ell\err} \, n_k^{\ell\err} \left(E_4^{\ell+1} - E_4^{\ell}\right) }{\halb \left(\left( \rho\varphi\right)^{\ell}+\left( \rho\varphi\right)^{\ell+1}  \right) \left(\varphi^{\ell+1}-\varphi^{\ell} \right) }, \nonumber \\ 
    \varphi^{\ell\err}   = \frac{\left(  \rho^{\ell} \varphi^{\ell} + \rho^{\err} \varphi^{\err}  \right)}{\left( \rho^{\ell} + \rho^{\err}\right)},\quad \rho^{\ell\err} = \halb \left(\rho^{\ell}+\rho^{\err}\right).\label{eqn.endscheme2d}
\end{gather}
The normal vectors and the fluctuations satisfy the properties:
\begin{equation}
    n_k^{\err\ell} = -n_k^{\ell\err}, \qquad \qquad 
	\nonumber \\
	\left(f_{\q,\,k}^{\err\ell}-f_{\q,\,k}^{\err}\right)n_k^{\err\ell}
	=\left(f_{\q,\,k}^{\err}-f_{\q,\,k}^{\ell\err}\right)n_k^{\ell\err}.
\end{equation}

In the following, we prove that the above semi-discrete finite volume scheme \eqref{eqn.schemeMHD2D} is thermodynamically compatible in the sense that it satisfies a cell-entropy inequality and a total energy conservation law as a consequence of the compatible discretization of all the other equations.  
\begin{theorem}
	\label{thm.entropy} 
	Assuming a positive temperature $T^{\ell} > 0$ and an at least positive semi-definite Hessian of the energy potential,  $\boldsymbol{\mathcal{H}}^{\ell\err}=\partial^2_{\q \q} \mathcal{E}^{\ell\err} \geq 0$, then the semi-discrete finite volume scheme \eqref{eqn.schemeMHD2D} above satisfies the cell entropy inequality 
	\begin{equation}
		\frac{\partial \rho S^{\ell}}{\partial t}  + \sum_{\err\in N_{\ell}}
		\frac{\left|\partial\Omega^{\ell\err}\right|}{\left|\Omega^{\ell}\right|}  D_{\rho S}^{\ell \err,-}		
		{-  \sum_{\err\in N_{\ell}}  \frac{\left|\partial\Omega^{\ell\err}\right|}{\left|\Omega^{\ell}\right|} g_{\rho S,\,\normal}^{\ell\err}}
		\geq 0.   \label{eqn.schemeGPR2D_entropy.ineq}
	\end{equation} 
\end{theorem}
\begin{proof}
	The proof is an immediate consequence of the discretization \eqref{eqn.schemeMHD2D_entropy}:  
	\begin{equation}
		\frac{\partial \rho S^{\ell}}{\partial t}  + \sum_{\err\in N_{\ell}}
		\frac{\left|\partial\Omega^{\ell\err}\right|}{\left|\Omega^{\ell}\right|}  D_{\rho S}^{\ell \err,-}
		{-  \sum_{\err\in N_{\ell}}  \frac{\left|\partial\Omega^{\ell\err}\right|}{\left|\Omega^{\ell}\right|} g_{\rho S,\,\normal}^{\ell\err}} 
		= \textcolor{blue}{\frac{1}{\left|\Omega^{\ell}\right|} \sum_{\err\in N_{\ell}} 
			\left|\partial\Omega^{\ell\err}\right| \Pi^{\ell\err,-}_{\normal}} \geq 0,  \label{eqn.schemeGPR2D_entropy2}
	\end{equation} 	
	since $\Pi^{\ell\err,-}_{\normal} = \frac{1}{2}\epsilon^{\ell\err} \frac{\Delta \q^{\ell\err}}{T^{\ell}} \boldsymbol{\mathcal{H}}^{\ell\err} \frac{\Delta \q^{\ell\err}}{\delta^{\ell\err}} \geq 0$ and $T^{\ell} > 0$.  		
\end{proof}

\begin{theorem}\label{theorem.energy.semidiscrete}
	The semi-discrete finite volume scheme \eqref{eqn.schemeMHD2D} admits the following total energy conservation law
	\begin{equation}\label{eqn.schemeMHD2D_energy}
		\frac{\partial \mathcal{E}^{\ell}}{\partial t} = 
		-\frac{1}{\left|\Omega^{\ell}\right|} \sum_{\err\in N_{\ell}} \left|\partial\Omega^{\ell\err}\right| D_{\mathcal{E}}^{\ell \err,-}
		\textcolor{blue}{+ \frac{1}{\left|\Omega^{\ell}\right|} \sum_{\err\in N_{\ell}} \left|\partial\Omega^{\ell\err}\right| g_{\mathcal{E},\,\normal}^{\ell\err}}
	\end{equation}	
	with
	\begin{equation}
		 D_{\mathcal{E}}^{\ell \err,-} +  D_{\mathcal{E}}^{ \err \ell,-} =  F^{\err} -F^{\ell},  \label{eqn.fluct_flux}
	\end{equation}
	and the discrete total energy flux in the normal direction  
	\begin{equation}
		F^{\ell} = F^{\ell}_{G} + \textcolor{red}{\left(  v_{k}^{\ell}\mathcal{E}^{\ell}_{3}	 
		 		+ v_{k}^{\ell} \mathcal{E}^{\ell}_{4} 
		 		+ v_{i}^{\ell} R_{ik}^{\ell}
		 		+ v_{k}^{\ell} \mu^{\ell}
		 		+ c_h \left( \varphi B_k \right)^{\ell} \right)n_k^{\ell\err}  }, 
	\end{equation}
	where 
	\begin{equation}
		 F^{\ell}_{G} =  v_k^{\ell} ( \mathcal{E}_1^{\ell} + \mathcal{E}_2^{\ell} + p^{\ell} )  n_k^{\ell\err} 
	\end{equation}	
	is the total energy flux related to the inviscid compressible Euler subsystem. 
	Assuming that the jumps on the boundary vanish, the scheme is nonlinearly marginally stable in the energy norm, i.e. the scheme satisfies   
	\begin{equation}
		\int_{\Omega} \frac{\partial \mathcal{E}^{\ell}}{\partial t} dV
		= \sum_{\ell} | \Omega^{\ell} | \frac{\partial \mathcal{E}^{\ell}}{\partial t} = 0.
	\end{equation}
\end{theorem}
\begin{proof}
	We first introduce the vector of thermodynamic dual variables 
	$$ \p^{\ell} = \partial_{\q} \mathcal{E}^{\ell}= \left( \partial_{\rho} \mathcal{E}^{\ell}, \partial_{\rho v_{i}} \mathcal{E}^{\ell}, \partial_{\rho S} \mathcal{E}^{\ell}, \partial_{B_{i}} \mathcal{E}^{\ell}, \partial_{\varphi} \mathcal{E}^{\ell} \right)^{T} := 
	\left( r^{\ell}, v_i^{\ell}, T^{\ell}, \beta_i^{\ell}, \psi^{\ell} \right)^{T}.$$ The dot product of $\p^{\ell}$ with the time derivatives of the conservative variables in \eqref{eqn.schemeMHD2D} yields,
	\begin{gather}\label{eqn.Etime_2d}
		\partial_{\rho} \mathcal{E}^{\ell} \frac{\partial \rho^{\ell}}{\partial t} 
		+ \partial_{\rho v_{i}} \mathcal{E}^{\ell} \frac{\partial \rho v_{i}^{\ell}}{\partial t}
		+ \partial_{\rho S} \mathcal{E}^{\ell} \frac{\partial \rho S^{\ell}}{\partial t} 
		+ \partial_{B_i} \mathcal{E}^{\ell} \frac{\partial B_{i}^{\ell}}{\partial t}
		+ \partial_{\phi} \mathcal{E}^{\ell} \frac{\partial \phi^{\ell}}{\partial t} 
		= \partial_{\q} \mathcal{E}^{\ell} \frac{\partial \q^{\ell}}{\partial t}
		= \frac{\partial \mathcal{E}^{\ell}}{\partial t}.
	\end{gather}
	which is the sought time derivative of the total energy. We now calculate the dot product of $\p^{\ell}$ with the inviscid terms of the scheme, which yields the following fluctuations of the total energy: 
\begin{eqnarray}\label{eqn.Efluc_2d}
	D_{\mathcal{E}}^{\ell \err,-} & =  &\partial_{\rho}\mathcal{E}^{\ell}_{1} D_{\rho}^{\ell \err,-}
	+ \partial_{\rho}\mathcal{E}^{\ell}_{2} D_{\rho}^{\ell \err,-}
	\textcolor{red}{+ \partial_{\rho}\mathcal{E}^{\ell}_{3} D_{\rho}^{\ell \err,-}}
	\textcolor{red}{+ \partial_{\rho}\mathcal{E}^{\ell}_{4} D_{\rho}^{\ell \err,-}}
	\nonumber\\ 
	&& 
	+ \partial_{\rho v_{i}}\mathcal{E}^{\ell} D_{\rho v_{i}}^{\ell \err,-} 	
	\textcolor{red}{
		+ \partial_{\rho v_{i}} \mathcal{E}^{\ell} R_{ik}^{\ell\err,-}  n_k^{\ell\err} 
		+ \partial_{\rho v_{i}} \mathcal{E}^{\ell} \mu^{\ell\err,-} n_i^{\ell\err}  
		 } 
	+ \partial_{\rho S}\mathcal{E}^{\ell} D_{\rho S}^{\ell \err,-}	
	\nonumber\\ 
	&& 
	  \textcolor{red}{+ 
	 	\partial_{B_{i}}\mathcal{E}^{\ell} \left( (B_i v_k)^{\ell\err} - (B_i v_k)^{\ell} \right) 	n_k^{\ell\err} }
 	\textcolor{red}{- 
 		\partial_{B_{i}}\mathcal{E}^{\ell} \left( (v_i B_k)^{\ell\err} - (v_i B_k)^{\ell} \right) 	n_k^{\ell\err} } 
	\nonumber\\ 
	&&
	\textcolor{red}{+ \partial_{B_{i}}\mathcal{E}^{\ell} \, v_i^{\ell\err} \frac{1}{2} \left( B_k^{\err} - B_k^{\ell} \right) n_k^{\ell\err} } 
	\textcolor{red}{+ \partial_{B_{i}}\mathcal{E}^{\ell} \, c_h \left( \varphi^{\ell\err} - \varphi^{\ell} \right) n_i^{\ell\err} }  
	\nonumber \\ && 
	\textcolor{red}{+ \partial_{\varphi}\mathcal{E}^{\ell} \frac{1}{2} \tilde{u}^{\ell\err}   \left(\varphi^{\err}-\varphi^{\ell}\right) }  
	\textcolor{red}{+ \partial_{\varphi}\mathcal{E}^{\ell} 
		\frac{1}{2} \frac{c_h}{\rho^{\ell\err}} \left(B_{k}^{\err}-B_{k}^{\ell} \right) n_k^{\ell\err} }.  
\end{eqnarray}	

Next, we prove that the sum of the obtained fluctuations can also be expressed as a difference of fluxes \eqref{eqn.fluct_flux}. We first focus on the compressible Euler subsystem using the discrete Godunov formalism (black terms):
\begin{eqnarray}
	\partial_{\rho}\mathcal{E}^{\ell}_{1} D_{\rho}^{\ell \err,-}
	+ \partial_{\rho}\mathcal{E}^{\ell}_{2} D_{\rho}^{\ell \err,-}		
	+ \partial_{\rho v_{i}}\mathcal{E}^{\ell} D_{\rho v_{i}}^{\ell \err,-}
	+ \partial_{\rho S}\mathcal{E}^{\ell} D_{\rho S}^{\ell \err,-} \nonumber \\
	+ \partial_{\rho}\mathcal{E}^{\err}_{1} D_{\rho}^{ \err\ell,-}
	+ \partial_{\rho}\mathcal{E}^{\err}_{2} D_{\rho}^{ \err\ell,-}		
	+ \partial_{\rho v_{i}}\mathcal{E}^{\err} D_{\rho v_{i}}^{ \err\ell,-} 
	+ \partial_{\rho S}\mathcal{E}^{\err} D_{\rho S}^{ \err\ell,-}
	\nonumber \\
	=\p^{\ell} \cdot \left(f_{\q,\,k}^{\ell\err}-f_{\q,\,k}^{\ell}\right)n_k^{\ell\err}
	+ \p^{\err} \cdot \left(f_{\q,\,k}^{\err\ell}-f_{\q,\,k}^{\err}\right)n_k^{\err\ell}
	\nonumber \\
	=\p^{\ell} \cdot \left(f_{\q,\,k}^{\ell\err}-f_{\q,\,k}^{\ell}\right)n_k^{\ell\err}
	+ \p^{\err} \cdot \left(f_{\q,\,k}^{\err}-f_{\q,\,k}^{\ell\err}\right)n_k^{\ell\err}
	\nonumber \\
	= -\left(\p^{\err} -\p^{\ell}\right) \cdot f_{\q,\,k}^{\ell\err}n_k^{\ell\err} + \p^{\err} \cdot 
	f_{\q,\,k}^{\err}n_k^{\ell\err} - \p^{\ell} \cdot f_{\q,\,k}^{\ell}n_k^{\ell\err}
	\nonumber \\
	= \left( \p^{\err} \cdot f_{\q,\,k}^{\err} - (v_{k}L)^{\err} \right) n_k^{\ell\err} 
	-\left( \p^{\ell} \cdot f_{\q,\,k}^{\ell} - (v_{k}L)^{\ell} \right) n_k^{\ell\err}
	= F^{\err}_{G} - F^{\ell}_{G}, \label{eqn.mstabilityblack_2D}
\end{eqnarray}
On the other hand, for the red terms, we get
\begin{eqnarray}
	\partial_{\rho}\mathcal{E}^{\ell}_{3} D_{\rho}^{\ell \err,-}+ \partial_{\rho}\mathcal{E}^{\ell}_{4} D_{\rho}^{\ell \err,-}
	+ \partial_{\rho v_{i}} \mathcal{E}^{\ell} R_{ik}^{\ell\err,-}  n_k^{\ell\err} 
	+ \partial_{\rho v_{i}} \mathcal{E}^{\ell} \mu^{\ell\err,-} n_i^{\ell\err}  
	\nonumber\\ 
	+ \partial_{B_{i}}\mathcal{E}^{\ell} \left( (B_i v_k)^{\ell\err} - (B_i v_k)^{\ell} \right) 	n_k^{\ell\err} 
	- \partial_{B_{i}}\mathcal{E}^{\ell} \left( (v_i B_k)^{\ell\err} - (v_i B_k)^{\ell} \right) 	n_k^{\ell\err}  	\nonumber\\ 
	+ \partial_{B_{i}}\mathcal{E}^{\ell} \, v_i^{\ell\err} \frac{1}{2} \left( B_k^{\err} - B_k^{\ell} \right) n_k^{\ell\err} 
	+ \partial_{B_{i}}\mathcal{E}^{\ell} \, c_h \left( \varphi^{\ell\err} - \varphi^{\ell} \right) n_i^{\ell\err} 	\nonumber \\ 
	+ \partial_{\varphi}\mathcal{E}^{\ell} \frac{1}{2} \tilde{u}^{\ell\err}   \left(\varphi^{\err}-\varphi^{\ell}\right) 
	+ \partial_{\varphi}\mathcal{E}^{\ell} \frac{1}{2} \frac{c_h}{\rho^{\ell\err}} \left(B_{k}^{\err}-B_{k}^{\ell} \right) n_k^{\ell\err} \nonumber\\
	\partial_{\rho}\mathcal{E}^{\err}_{3} D_{\rho}^{\err\ell,-}+ \partial_{\rho}\mathcal{E}^{\err}_{4} D_{\rho}^{\err\ell,-}
	+ \partial_{\rho v_{i}} \mathcal{E}^{\err} R_{ik}^{\err\ell,-}  n_k^{\err\ell} 
	+ \partial_{\rho v_{i}} \mathcal{E}^{\err} \mu^{\err\ell,-} n_i^{\err\ell}   
	\nonumber\\ 
	+ \partial_{B_{i}}\mathcal{E}^{\err} \left( (B_i v_k)^{\err\ell} - (B_i v_k)^{\err} \right) 	n_k^{\err\ell} 
	- \partial_{B_{i}}\mathcal{E}^{\err} \left( (v_i B_k)^{\err\ell} - (v_i B_k)^{\err} \right) 	n_k^{\err\ell}  	\nonumber\\ 
	+ \partial_{B_{i}}\mathcal{E}^{\err} \, v_i^{\err\ell} \frac{1}{2} \left( B_k^{\ell} - B_k^{\err} \right) n_k^{\err\ell} 
	+ \partial_{B_{i}}\mathcal{E}^{\err} \, c_h \left( \varphi^{\err\ell} - \varphi^{\err} \right) n_i^{\err\ell} 	\nonumber \\ 
	+ \partial_{\varphi}\mathcal{E}^{\err} \frac{1}{2} \tilde{u}^{\err\ell}   \left(\varphi^{\ell}-\varphi^{\err}\right) 
	+ \partial_{\varphi}\mathcal{E}^{\err} \frac{1}{2} \frac{c_h}{\rho^{\err\ell}} \left(B_{k}^{\ell}-B_{k}^{\err} \right) n_k^{\err\ell} \nonumber\\
	= E^{\ell}_{4} \left(f_{\rho,\,k}^{\ell\err}-f_{\rho,\,k}^{\ell}\right)n_k^{\ell\err}
	+ v_{i}^{\ell} \left(R_{ik}^{\ell\err}- R_{ik}^{\ell}\right) n_k^{\ell\err} 
	+ v_{i}^{\ell} \left(\mu^{\ell\err}- \mu^{\ell}\right) n_i^{\ell\err}  
	\nonumber\\ 
	+ B_{i}^{\ell} \left( (B_i v_k)^{\ell\err} - (B_i v_k)^{\ell} \right) 	n_k^{\ell\err} 
	- B_{i}^{\ell} \left( (v_i B_k)^{\ell\err} - (v_i B_k)^{\ell} \right) 	n_k^{\ell\err}  	\nonumber\\ 
	+ B_{i}^{\ell} \, v_i^{\ell\err} \frac{1}{2} \left( B_k^{\err} - B_k^{\ell} \right) n_k^{\ell\err} 
	+ B_{i}^{\ell} \, c_h \left( \varphi^{\ell\err} - \varphi^{\ell} \right) n_i^{\ell\err} 	\nonumber \\ 
	+ \psi^{\ell} \frac{1}{2} \tilde{u}^{\ell\err}   \left(\varphi^{\err}-\varphi^{\ell}\right) 
	+ \psi^{\ell} \frac{1}{2} \frac{c_h}{\rho^{\ell\err}} \left(B_{k}^{\err}-B_{k}^{\ell} \right) n_k^{\ell\err} \nonumber\\
	+E^{\err}_{4} \left(f_{\rho,\,k}^{\err}- f_{\rho,\,k}^{\ell\err}\right)n_k^{\ell\err}
	+ v_{i}^{\err} \left( R_{ik}^{\err} - R_{ik}^{\ell\err}  \right) n_k^{\ell\err} 
	+ \rho v_{i}^{\err} \left( \mu^{\err} - \mu^{\ell\err}  \right) n_i^{\ell\err}  
	\nonumber\\ 
	+  B_{i}^{\err} \left((B_i v_k)^{\err} - (B_i v_k)^{\ell\err}\right) 	n_k^{\ell\err} 
	- B_{i}^{\err} \left( (v_i B_k)^{\err}- (v_i B_k)^{\ell\err} \right) 	n_k^{\ell\err}  	\nonumber\\ 
	+ B_{i}^{\err} \, v_i^{\ell\err} \frac{1}{2} \left(B_k^{\err} - B_k^{\ell}\right) n_k^{\ell\err} 
	+ B_{i}^{\err} \, c_h \left( \varphi^{\err} - \varphi^{\ell\err}  \right) n_i^{\ell\err} 	\nonumber \\ 
	+ \psi^{\err} \frac{1}{2} \tilde{u}^{\ell\err}   \left(\varphi^{\err}-\varphi^{\ell}\right) 
	+ \psi^{\err} \frac{1}{2} \frac{c_h}{\rho^{\ell\err}} \left(B_{k}^{\err}-B_{k}^{\ell} \right) n_k^{\ell\err} \nonumber\\
	=
	\left(  v_{k}^{\err}\mathcal{E}^{\err}_{3}	 
	+ v_{k}^{\err} \mathcal{E}^{\err}_{4} 
	+ v_{i}^{\err} R_{ik}^{\err}
	+ v_{k}^{\err} \mu^{\err}
	+ c_h \left( \varphi B_k \right)^{\err} \right)n_k^{\ell\err} \nonumber \\   
	-\left(  v_{k}^{\ell}\mathcal{E}^{\ell}_{3}	 
	+ v_{k}^{\ell} \mathcal{E}^{\ell}_{4} 
	+ v_{i}^{\ell} R_{ik}^{\ell}
	+ v_{k}^{\ell} \mu^{\ell}
	+ c_h \left( \varphi B_k \right)^{\ell} \right)n_k^{\ell\err},
	\label{eqn.mstabilityred_2D}
\end{eqnarray}
after taking into account the expression of $\partial_{\q} \mathcal{E}$ in state variables, that $\partial_{\rho} \mathcal{E}_{3}=0$  and the definitions
\eqref{eqn.fluctuations_2D}-\eqref{eqn.endscheme2d}.
Combining \eqref{eqn.mstabilityblack_2D}-\eqref{eqn.mstabilityred_2D} leads to the sought result, \eqref{eqn.fluct_flux},
		$D_{\mathcal{E}}^{\ell \err,-} +  D_{\mathcal{E}}^{\err\ell,-} = F^{\err} -F^{\ell}$.

To demonstrate that the energy conservation law \eqref{eqn.schemeMHD2D_energy} is retrieved, we start performing some algebraic manipulations on the numerical diffusion terms \eqref{eqn.diffusion_2D} and we use \eqref{eqn.E.diss.2} and \eqref{eqn.roeprop2} leading to 
\begin{eqnarray} 
	\frac{1}{\left|\Omega^{\ell}\right|}\! \sum_{\err\in N_{\ell}} \! \left|\partial\Omega^{\ell\err}\right|\left( \p^{\ell} \cdot \mathbf{P}^{\ell\err,-}_{\normal} + \p^{\ell} \cdot \g^{\ell\err}_{\normal}\right)  =
	\sum_{\err\in N_{\ell}} \!\frac{\left|\partial\Omega^{\ell\err}\right|}{\left|\Omega^{\ell}\right|}\left( \p^{\ell} \cdot \mathbf{P}^{\ell\err,-}_{\normal}+ \p^{\ell} \cdot
	\epsilon^{\ell\err}\frac{\Delta \q^{\ell\err}}{\delta^{\ell\err}}\right)  
	\nonumber\\
	\!=\!\sum_{\err\in N_{\ell}} \! \frac{\left|\partial\Omega^{\ell\err}\right|}{\left|\Omega^{\ell}\right|} \! \left( \p^{\ell}\! \cdot \mathbf{P}^{\ell\err,-}_{\normal}
	\!+\! \halb \p^{\ell}\! \cdot \epsilon^{\ell\err}\frac{\Delta \q^{\ell\err}}{\delta^{\ell\err}} 
	\!+\! \halb \p^{\err}\! \cdot \epsilon^{\ell\err}\frac{\Delta \q^{\ell\err}}{\delta^{\ell\err}} 
	\!+\! \halb\p^{\ell}\! \cdot \epsilon^{\ell\err}\frac{\Delta \q^{\ell\err}}{\delta^{\ell\err}} 
	\!-\! \halb \p^{\err}\! \cdot \epsilon^{\ell\err}\frac{\Delta \q^{\ell\err}}{\delta^{\ell\err}}\right)
	\nonumber\\
	=\!\sum_{\err\in N_{\ell}} \!\frac{\left|\partial\Omega^{\ell\err}\right|}{\left|\Omega^{\ell}\right|}\left( \p^{\ell} \cdot \mathbf{P}^{\ell\err,-}_{\normal}
	+ \halb \left( \p^{\ell}+\p^{\err}\right)  \cdot \epsilon^{\ell\err}\frac{\Delta \q^{\ell\err}}{\delta^{\ell\err}} 
	- \halb\left(\p^{\err} -\p^{\ell}\right)  \cdot \epsilon^{\ell\err}\frac{\Delta \q^{\ell\err}}{\delta^{\ell\err}} \right) \nonumber \\
	=\!\sum_{\err\in N_{\ell}} \!\frac{\left|\partial\Omega^{\ell\err}\right|}{\left|\Omega^{\ell}\right|}\left( \p^{\ell} \cdot \mathbf{P}^{\ell\err,-}_{\normal}
	+ \epsilon^{\ell\err}\frac{\Delta \mathcal{E}^{\ell\err}}{\delta^{\ell\err}} 
	- \halb\epsilon^{\ell\err}\frac{\Delta \q^{\ell\err}}{\delta^{\ell\err}} 
	\partial^2_{\q\q}\mathcal{E}^{\ell\err}\Delta \q^{\ell\err} \right)\!. 
	\label{eqn.diss.2d} 		 
\end{eqnarray}
Taking into account $\mathbf{P}^{\ell\err,-}_{\normal} = \left(0,\mathbf{0},\Pi^{\ell\err,-}_{\normal},\mathbf{0},0\right)$  and \eqref{eqn.production_2D} gives
\begin{equation}\label{eqn.diff_2D}
	\frac{1}{\left|\Omega^{\ell}\right|}\! \sum_{\err\in N_{\ell}}   \!  \left|\partial\Omega^{\ell\err}\right| \left( \p^{\ell} \cdot  \g_{\normal}^{\ell\err} +  \p^{\ell} \cdot \mathbf{P}^{\ell\err,-}_{\normal} \right) 		 
	= \frac{1}{\left|\Omega^{\ell}\right|} \!\sum_{\err\in N_{\ell}}\! \left|\partial\Omega^{\ell\err}\right|  \epsilon^{\ell\err}\frac{\Delta \mathcal{E}^{\ell\err}}{\delta^{\ell\err}} 
	= \frac{1}{\left|\Omega^{\ell}\right|}\! \sum_{\err\in N_{\ell}}\! \left|\partial\Omega^{\ell\err}\right| g_{\mathcal{E},\,\normal}^{\ell\err} 
\end{equation}
and thus 
\begin{eqnarray}
	 \p^{\ell} \cdot \frac{1}{\left|\Omega^{\ell}\right|}\! \sum_{\err\in N_{\ell}} \!    \left|\partial\Omega^{\ell\err}\right| \g_{\normal}^{\ell\err}
	+\partial_{\rho S} \mathcal{E}^{\ell} \frac{1}{\left|\Omega^{\ell}\right|}\! \sum_{\err\in N_{\ell}}\! 
	\left|\partial\Omega^{\ell\err}\right| \Pi^{\ell\err}_{\normal}
	\\ 	 
	= \!\sum_{\err\in N_{\ell}} \! \frac{\left|\partial\Omega^{\ell\err}\right|}{\left|\Omega^{\ell}\right|}  \left( 
	\p^{\ell} \cdot \epsilon^{\ell\err}\frac{\Delta \q^{\ell\err}}{\delta^{\ell\err}}
	+ \partial_{\rho S} \mathcal{E}^{\ell}  \frac{1}{4}\epsilon^{\ell\err} \frac{\Delta \q^{\ell\err}}{T^{\ell}} \partial^2_{\q\q} \mathcal{E}^{\ell\err} \frac{\Delta \q^{\ell\err}}{\delta^{\ell\err}} \right)\!= \!\frac{1}{\left|\Omega^{\ell}\right|}\! \sum_{\err\in N_{\ell}}\! \left|\partial\Omega^{\ell\err}\right| g_{\mathcal{E},\,\normal}^{\ell\err}.
	\nonumber 
	\label{eqn.Ediff_2D}
\end{eqnarray}
From \eqref{eqn.Etime_2d}, \eqref{eqn.Efluc_2d}, \eqref{eqn.Ediff_2D}, we conclude that the thermodynamically compatible FV scheme satisfies the additional semi-discrete total energy conservation law that takes the sought form of \eqref{eqn.schemeMHD2D_energy}, i.e.   
\begin{equation*}
	\frac{\partial \mathcal{E}^{\ell}}{\partial t} = 
	-\frac{1}{\left|\Omega^{\ell}\right|} \sum_{\err\in N_{\ell}} \left|\partial\Omega^{\ell\err}\right| D_{\mathcal{E}}^{\ell \err,-}
	\textcolor{blue}{+ \frac{1}{\left|\Omega^{\ell}\right|} \sum_{\err\in N_{\ell}} \left|\partial\Omega^{\ell\err}\right| g_{\mathcal{E},\,\normal}^{\ell\err}}.
\end{equation*}

To complete the proof of the theorem, we consider marginal nonlinear stability in the energy norm. Integrating \eqref{eqn.schemeMHD2D_energy} over the computational domain $\Omega$ gives
\begin{equation*}
	\int_{\Omega} \frac{\partial \mathcal{E}^{\ell}}{\partial t} dV
	= \sum_{\ell} \left| \Omega^{\ell}\right| \frac{\partial \mathcal{E}^{\ell}}{\partial t}
	= -\sum_{\ell} \sum_{\err\in N_{\ell}} \left|\partial\Omega^{\ell\err}\right| D_{\mathcal{E}}^{\ell \err,-}
	\textcolor{blue}{+ \sum_{\ell} \sum_{\err\in N_{\ell}} \left|\partial\Omega^{\ell\err}\right| g_{\mathcal{E},\,\normal}^{\ell\err}}.
\end{equation*}
Assuming that the solution on $\partial \Omega$ tends to a constant value, we observe that the jumps on the state variables, $\q$, vanish at the boundary, that dissipative terms and fluctuations become zero and that we can recast the remaining dissipative terms into a telescopic sum which cancels.
Therefore, we can reorganize the summation of the first term in the right hand side of the previous equation gathering the contributions at each face so that
\begin{equation*}
	\int_{\Omega} \frac{\partial \mathcal{E}^{\ell}}{\partial t} dV\!
	=\! \sum_{\ell} \left| \Omega^{\ell}\right| \frac{\partial \mathcal{E}^{\ell}}{\partial t}\! = -\!\sum_{\ell\err} \left|\partial\Omega^{\ell\err}\right|  \left(D_{\mathcal{E}}^{\ell \err,-}
	+  D_{\mathcal{E}}^{ \err \ell,-}\right)\! = \! -\!\sum_{\ell\err}  \left|\partial\Omega^{\ell\err}\right|\left(  F^{\err}-F^{\ell}\right)\! = \! 0
\end{equation*}
which proves the marginal nonlinear stability. 
\end{proof}

\textcolor{black}{In the absence of numerical dissipation, the discretization of the Euler subsystem (black terms) and of the terms related to the magnetic field and to the cleaning scalar (red terms) is a central one. Therefore, on uniform Cartesian meshes we expect the scheme presented in this paper to be second order accurate in space, which is later also confirmed by numerical experiments.} 


\section{Numerical results}
\label{sec.results}

\textcolor{black}{In order to keep time discretization errors as small as possible},  throughout this section we use the classical fourth order Runge-Kutta scheme to discretize the nonlinear ODE system that results from the semi-discrete HTC scheme. 
\textcolor{black}{The time step size is chosen according to the following standard CFL-type condition
\begin{equation}
	\Delta t = \frac{\textnormal{CFL}}{\frac{|\lambda_{\max}^x|}{\Delta x} + \frac{|\lambda_{\max}^y|}{\Delta y}}
\end{equation}
with $|\lambda_{\max}^x|$ and $|\lambda_{\max}^y|$ the maximum absolute values of the eigenvalues in the $x$ and $y$ direction, respectively, and the Courant number CFL = 0.5. 
} 
If not stated otherwise, the numerical viscosity $\epsilon^{\ell+\halb}$ is chosen according to  \eqref{eqn.viscosity}. Wherever values of $\epsilon$ are explicitly provided, the numerical dissipation is set to a constant, $\epsilon^{\ell+\halb}=\epsilon$. The cleaning scalar is initially set to $\varphi=0$ and if not stated otherwise the cleaning speed is set to $c_h=2$.  
\subsection{Numerical convergence study} 
\label{sec.numconv} 
To verify the order of accuracy of the new HTC FV scheme for MHD, we provide a numerical convergence study using a smooth MHD vortex problem, similar to the one proposed in \cite{Balsara2004}. Here, we follow the setup given in \cite{SIMHD}, but adapt it to the unit system used in this paper. The initial condition, which is also the exact solution of the problem for all later times, is given in terms of primitive variables by $\rho = 1$, $v_1 = e^{\frac{1}{2}(1-r^2)}(5-y)$, 
$v_2 = e^{\frac{1}{2}(1-r^2)}(x-5)$, $v_3 = 0$, $p = \halb e - \halb r^2 e^{-(r^2-1)}$, $ B_1 = e^{\frac{1}{2}(1-r^2)}(5-y)$, $B_2 = e^{\frac{1}{2}(1-r^2)}(x-5)$ and $B_3 = 0$  
with $r^2=(x-5)^2+(y-5)^2$. 
 The parameters of the model are chosen as $\gamma=\frac{5}{3}$ and $\epsilon=0$. The computational domain $\Omega=[0,10]^2$ with periodic boundary conditions is discretized with a sequence of successively refined uniform Cartesian grids composed of $N_x \times N_x$ elements. The numerical convergence rates obtained 
at time $t=0.25$ are reported in Table \ref{tab.conv.semi}, showing that second order of accuracy is achieved by our scheme.  
\begin{table}[!t] 
	\renewcommand{\arraystretch}{1.1}
	\caption{ $L^2$ error norms for the smooth MHD vortex problem obtained with the semi-discrete HTC finite volume scheme at time $t=0.25$.  } 
	\label{tab.conv.semi} 
	\begin{center} 
		\begin{tabular}{ccccccccc}
			\hline 
			$N_{x}$  &  $\|\rho\|^2 $	& $\| \rho v_1 \|^2$ & $ \| \rho S \|^2 $ & $ \| B_1 \|^2 $  & $\mathcal{O}(\rho)$ & $\mathcal{O}(\rho v_1)$ & $\mathcal{O}(\rho S)$ & $\mathcal{O}(B_1)$\\ 
			\hline
			32  & 1.03E-2 & 1.13E-2 & 9.35E-3 & 7.61E-3 &  & &  &  \\
			64  & 2.72E-3 & 2.91E-3 & 2.36E-3 & 2.06E-3 & 1.9 & 2.0 & 2.0 & 1.9 \\
			128 & 6.91E-4 & 7.32E-4 & 5.90E-4 & 5.25E-4 & 2.0 & 2.0 & 2.0 & 2.0 \\
			256 & 1.73E-4 & 1.83E-4 & 1.47E-4 & 1.32E-4 & 2.0 & 2.0 & 2.0 & 2.0 \\
			512 & 4.34E-5 & 4.58E-5 & 3.68E-5 & 3.30E-5 & 2.0 & 2.0 & 2.0 & 2.0 \\
			\hline
		\end{tabular} 
	\end{center}
\end{table}

\textcolor{black}{We now run this test again on a mesh of $64 \times 64$ elements, but until a much larger final time of $t=50$, once with divergence cleaning ($c_h=2$) and once without divergence cleaning ($c_h=0$). We measure the $L^{\infty}$ norm of the divergence error of the magnetic field, as well as the integral of the entropy density over the domain. In Figure~\ref{fig.divBrhoS} we report the time series of the divergence error and of the entropy integral. We note that in both simulations the entropy is constant in time, while the divergence errors are more than two orders of magnitude smaller with the divergence cleaning, as expected. }   
\begin{figure}[!htbp]
	\begin{center}
		\begin{tabular}{cc} 
			\includegraphics[width=0.45\textwidth]{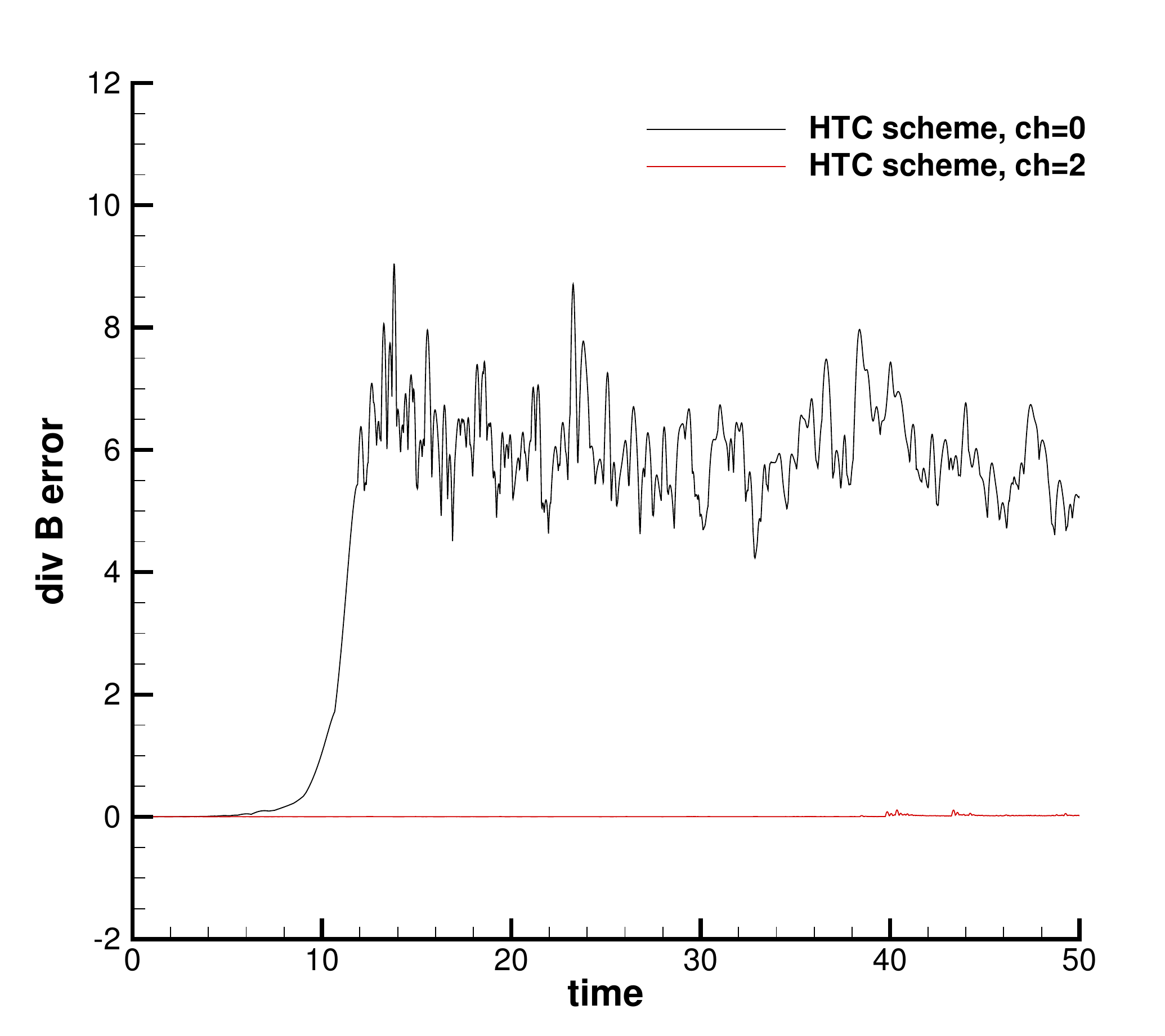} & 
			\includegraphics[width=0.45\textwidth]{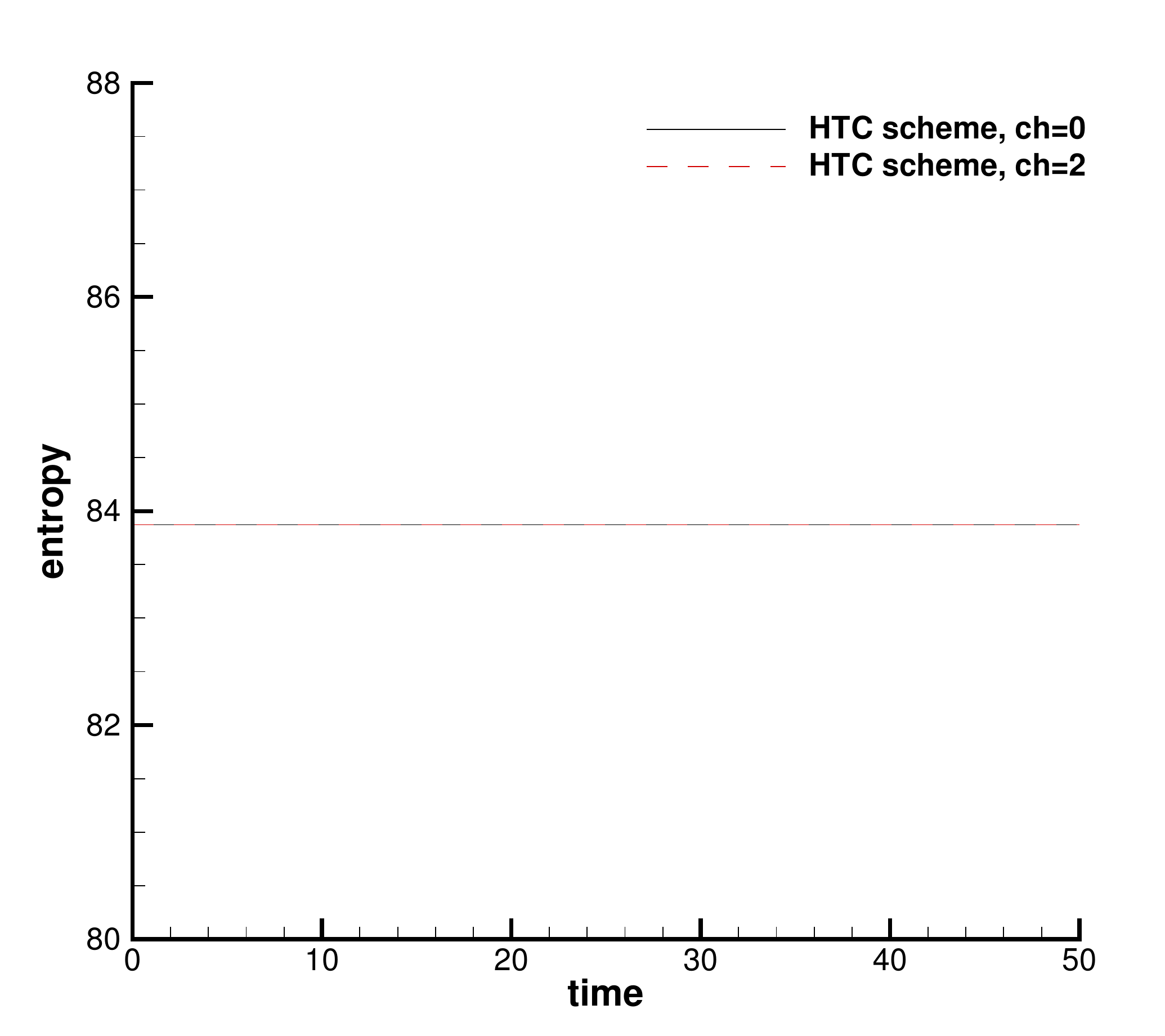}   
		\end{tabular} 
		\caption{\textcolor{black}{Time series of the $L^{\infty}$ norm of the divergence error (left) and of the integral of the entropy density over the domain (right). One simulation is carried out with divergence cleaning ($c_h=2$), while the other simulation does not employ any divergence cleaning ($c_h=0$).}  } 
		\label{fig.divBrhoS}
	\end{center}
\end{figure}
\subsection{Riemann problems} 
\label{sec.rp} 
In this section, we solve four Riemann problems of the ideal MHD equations using the new HTC finite volume scheme proposed in this paper. The setup follows the one given in \cite{SIMHD}. The exact Riemann solver has kindly been provided by S.A.E.G. Falle \cite{fallemhd,falle2}. The computational domain $\Omega=[-0.5,0.5]$ has been discretized with 1000 uniform control volumes. The initial condition consists in constant left, $L$, and right, $R$, states, separated by a discontinuity in $x_d$, with $x_d=0$ for RP1 and RP4, while $x_d=-0.1$ for RP2 and RP3. The initial values of density, velocity, pressure, and magnetic field are reported in Table \ref{tab.ic.mhd}. In all cases we set $\gamma=\frac{5}{3}$.   
The comparison between the numerical solution obtained with the new HTC FV scheme and the exact solution is presented in Figure \ref{fig.rp123}. A good agreement can be observed, similar to the results shown in \cite{SIMHD} and \cite{OsherUniversal}. 
\begin{table}[!htbp]
	\caption{Initial data for density $\rho$, velocity $\mathbf{v} = (u,v,w)$, pressure $p$  
		and magnetic field $\mathbf{B} = (B_x,B_y,B_z)$ for the Riemann problems of the ideal MHD equations. } 
	\begin{center} 
		\begin{tabular}{rcccccccc}
			\hline
			Case & $\rho$ & $u$ & $v$ & $w$ & $p$ & $B_x$ & $B_y$ & $B_z$        \\ 
			\hline   
			RP1 L: &  1.0    &  0.0     & 0.0    & 0.0      &  1.0     & $\frac{3}{4}$ &  $\phantom{-}1.0$  & 0.0       \\
			R: &  0.125  &  0.0     & 0.0    & 0.0      &  0.1     & $\frac{3}{4} $ & $-1.0$  & 0.0       \\
			RP2 L: &  1.08   &  1.2     & 0.01   & 0.5      &  0.95    & $\frac{2.0}{\sqrt{4 \pi}}$ &  $\frac{3.6}{\sqrt{4 \pi}}$     & $\frac{2.0}{\sqrt{4 \pi}}$            \\
			R: &  0.9891 &  -0.0131 & 0.0269 & 0.010037 &  0.97159 & $\frac{2.0}{\sqrt{4 \pi}}$ &  $\frac{4.0244}{\sqrt{4 \pi}}$  & $\frac{2.0026}{\sqrt{4 \pi}}$         \\
			RP3 L: &  1.7    &  0.0     & 0.0    & 0.0      &  1.7     & 1.1 &  1.0  & 0.0              \\
			R: &  0.2    &  0.0     & 0.0    & -1.49689  &  0.2   & 1.1 &  $\frac{2.7859}{\sqrt{4 \pi}}$  & $\frac{2.1921}{\sqrt{4 \pi}}$         \\
			RP4 L: &  1.0    &  0.0     & 0.0    & 0.0      &  1.0     & $1.3 $ &  $\phantom{-}1.0$   & 0.0            \\
			R: &  0.4    &  0.0     & 0.0    & 0.0      &  0.4     & $1.3 $ &  $-1.0$  & 0.0             \\
			\hline
		\end{tabular}
	\end{center} 
	\label{tab.ic.mhd}
\end{table} 
\begin{figure}[!htbp]
	\begin{center}
		\begin{tabular}{cc} 
			\includegraphics[width=0.4\textwidth]{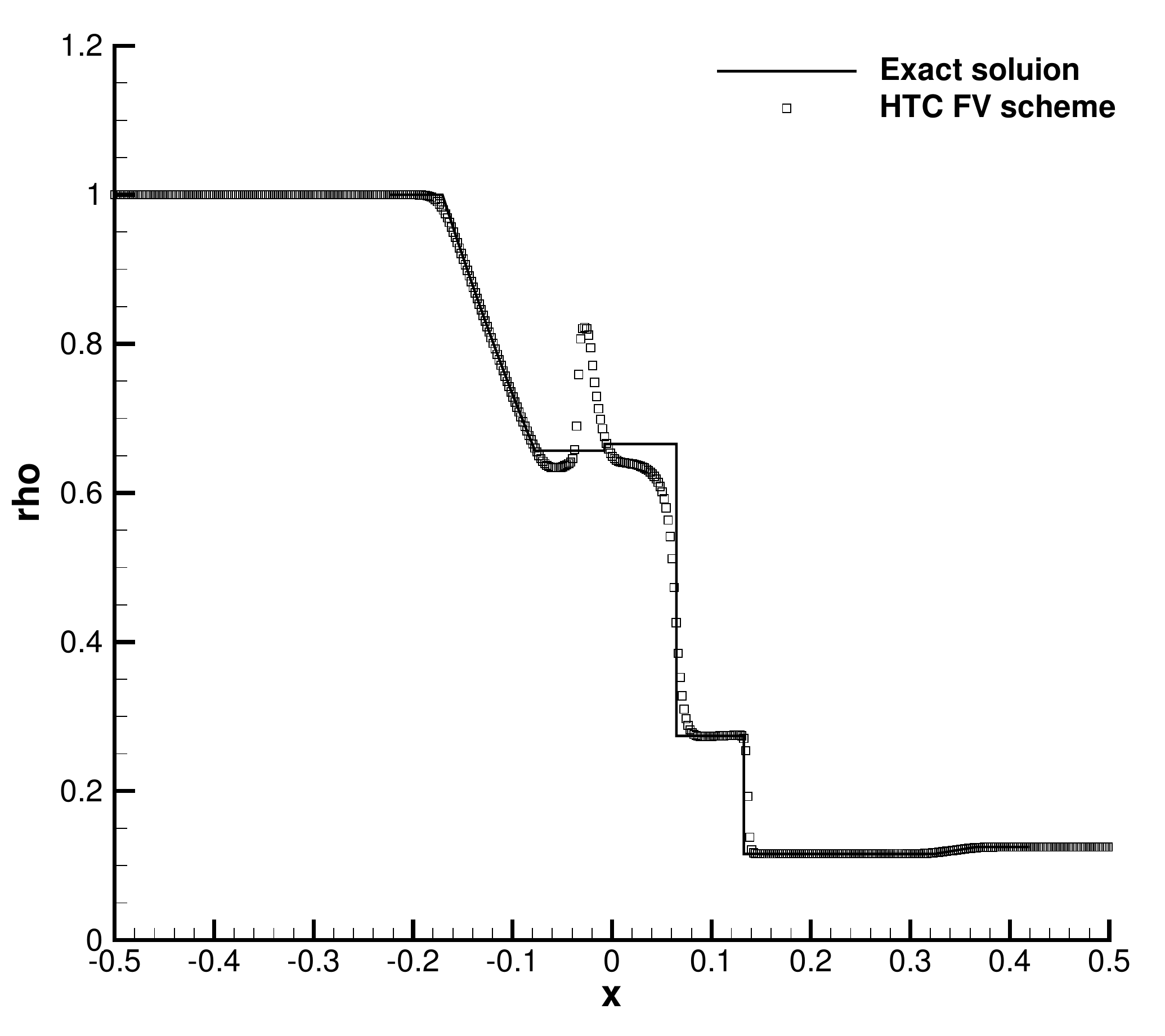} & 
			\includegraphics[width=0.4\textwidth]{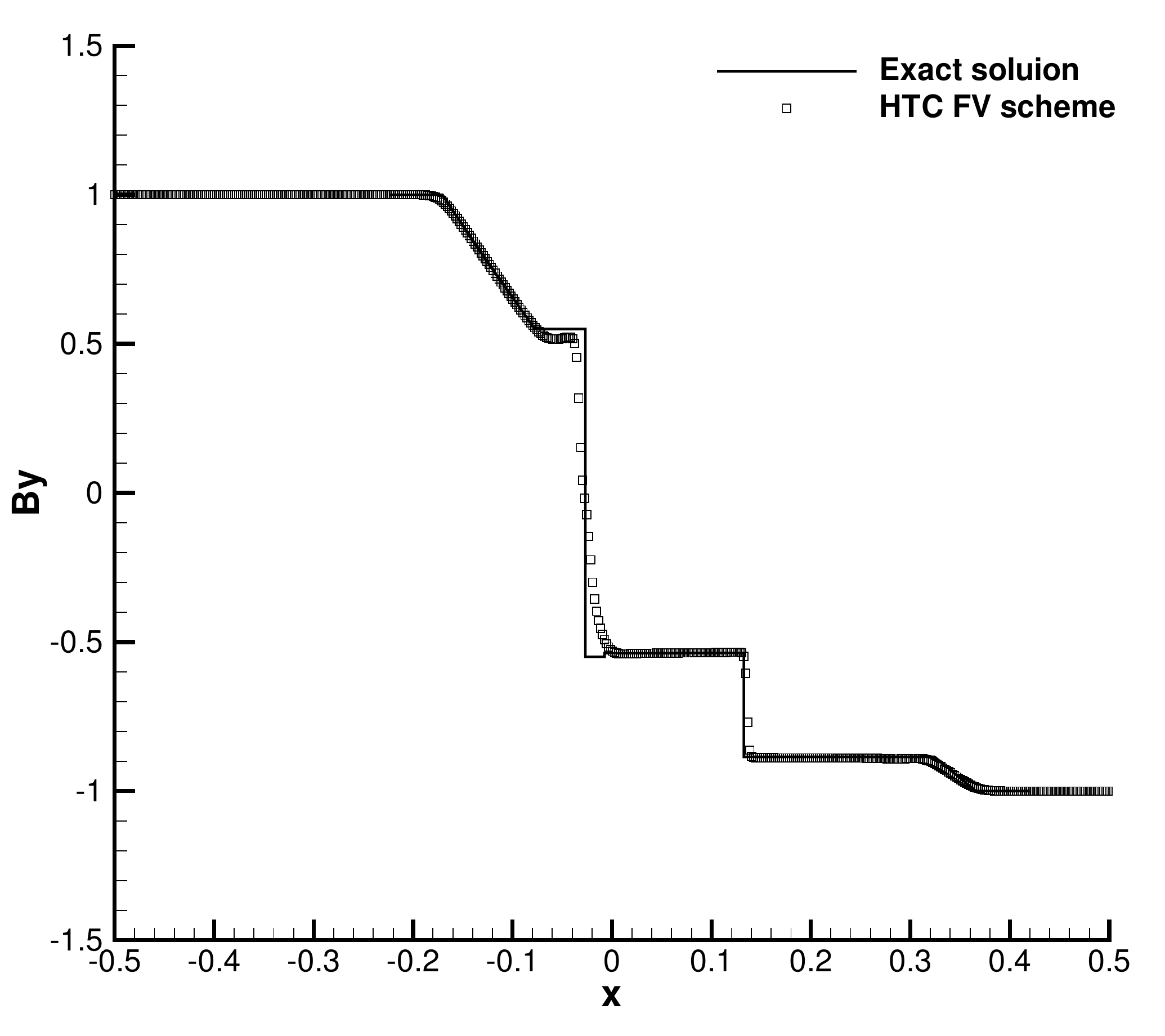}  \\  
			\includegraphics[width=0.4\textwidth]{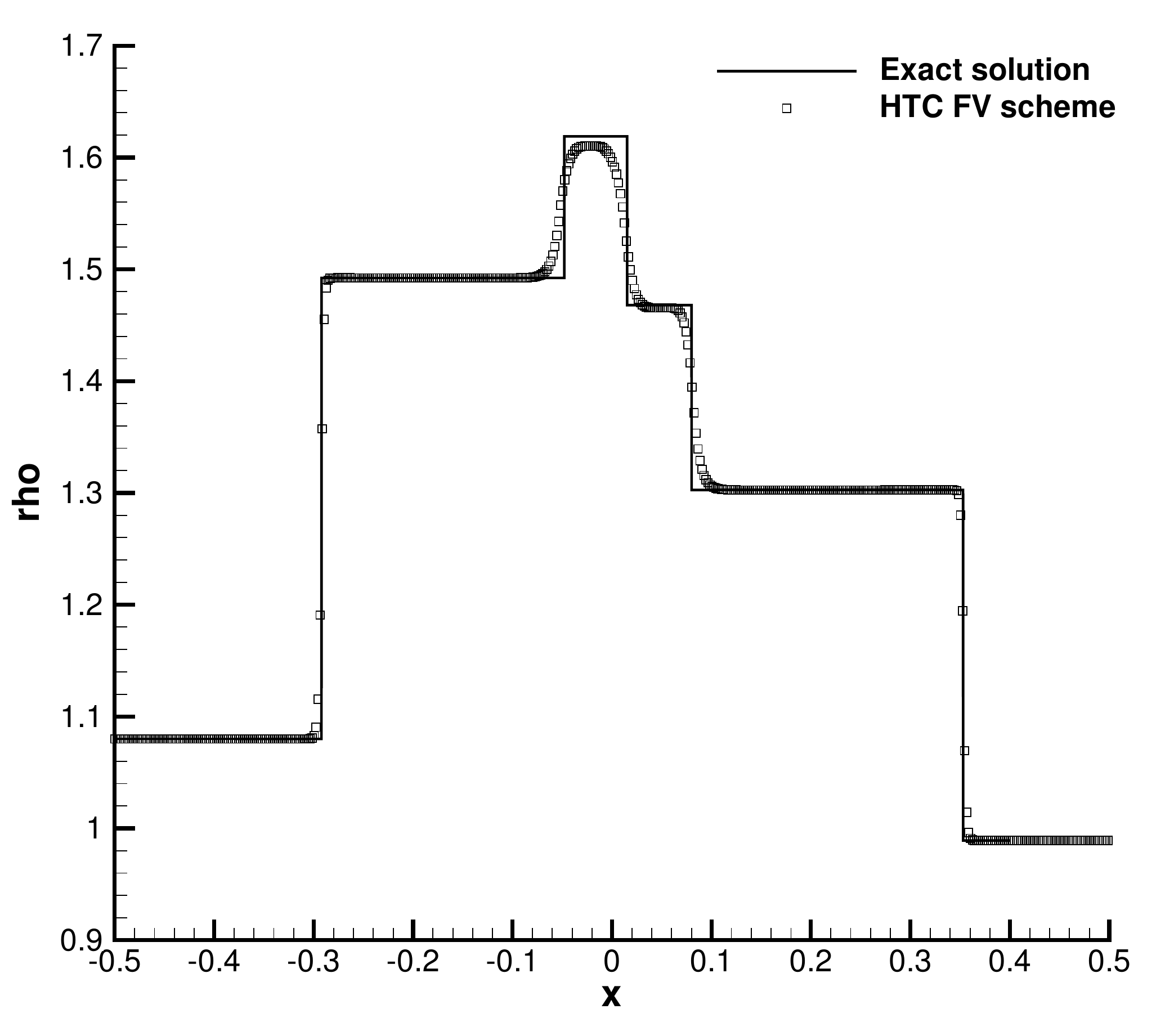} & 
			\includegraphics[width=0.4\textwidth]{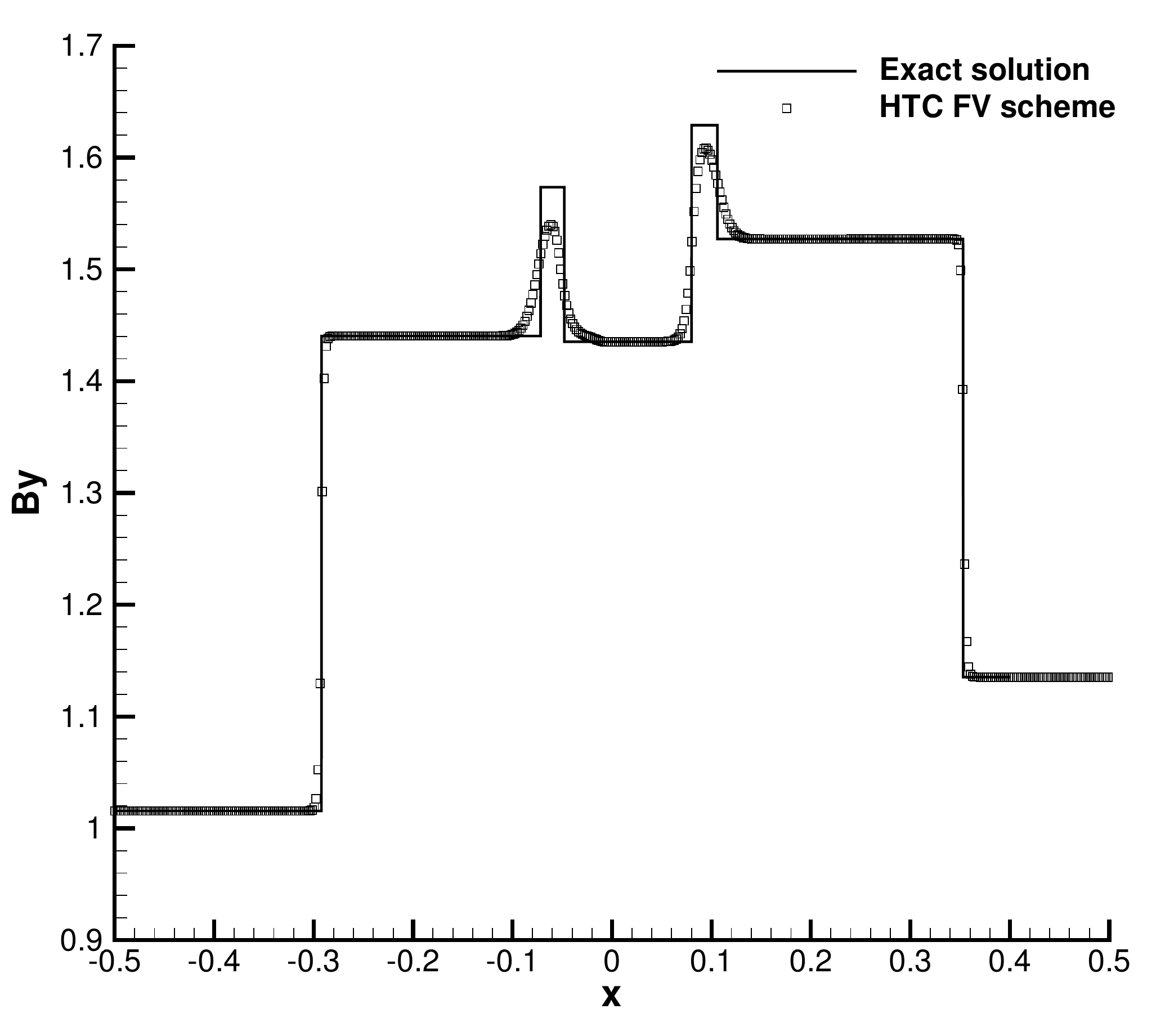}  \\   
			\includegraphics[width=0.4\textwidth]{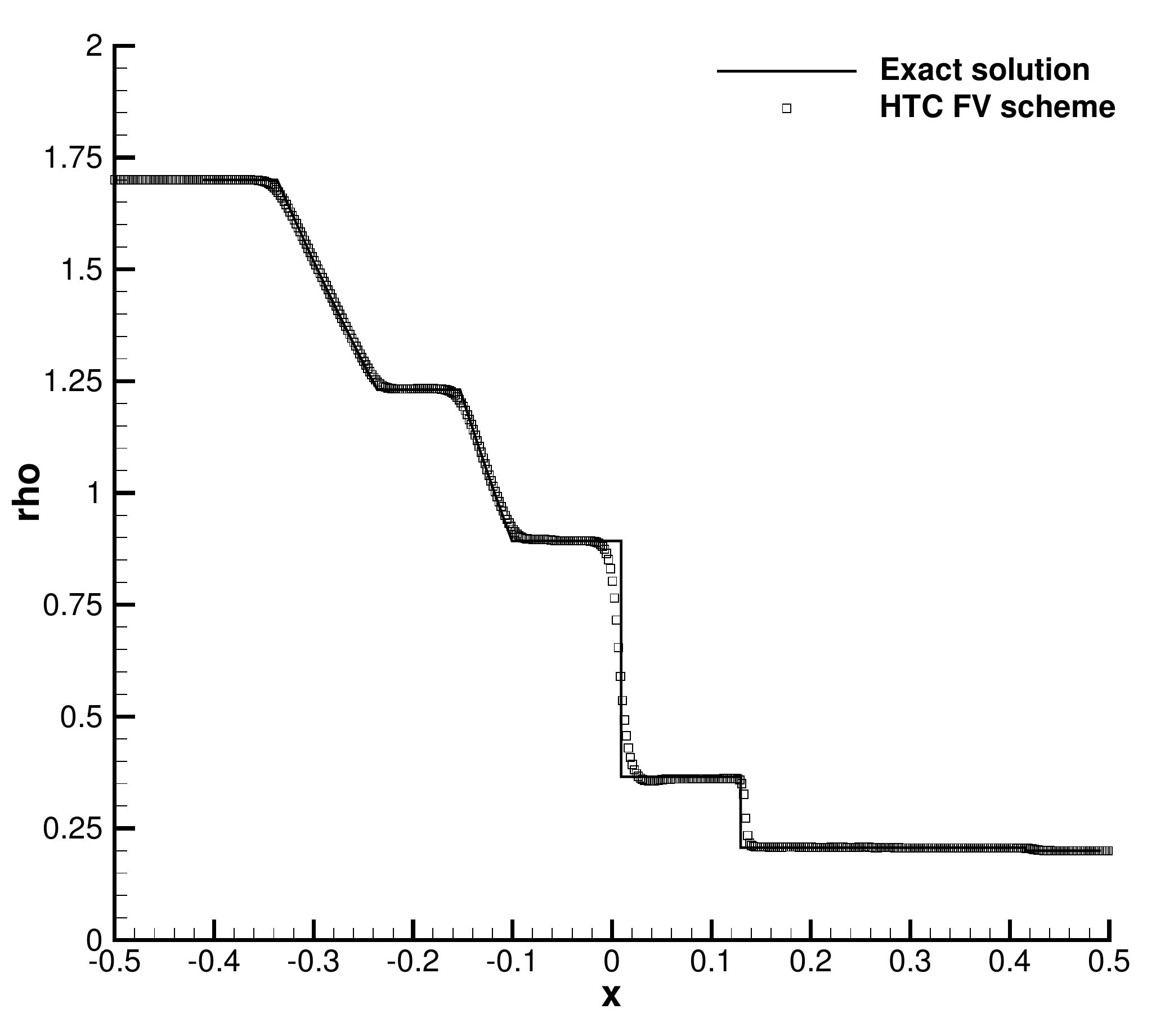} & 
			\includegraphics[width=0.4\textwidth]{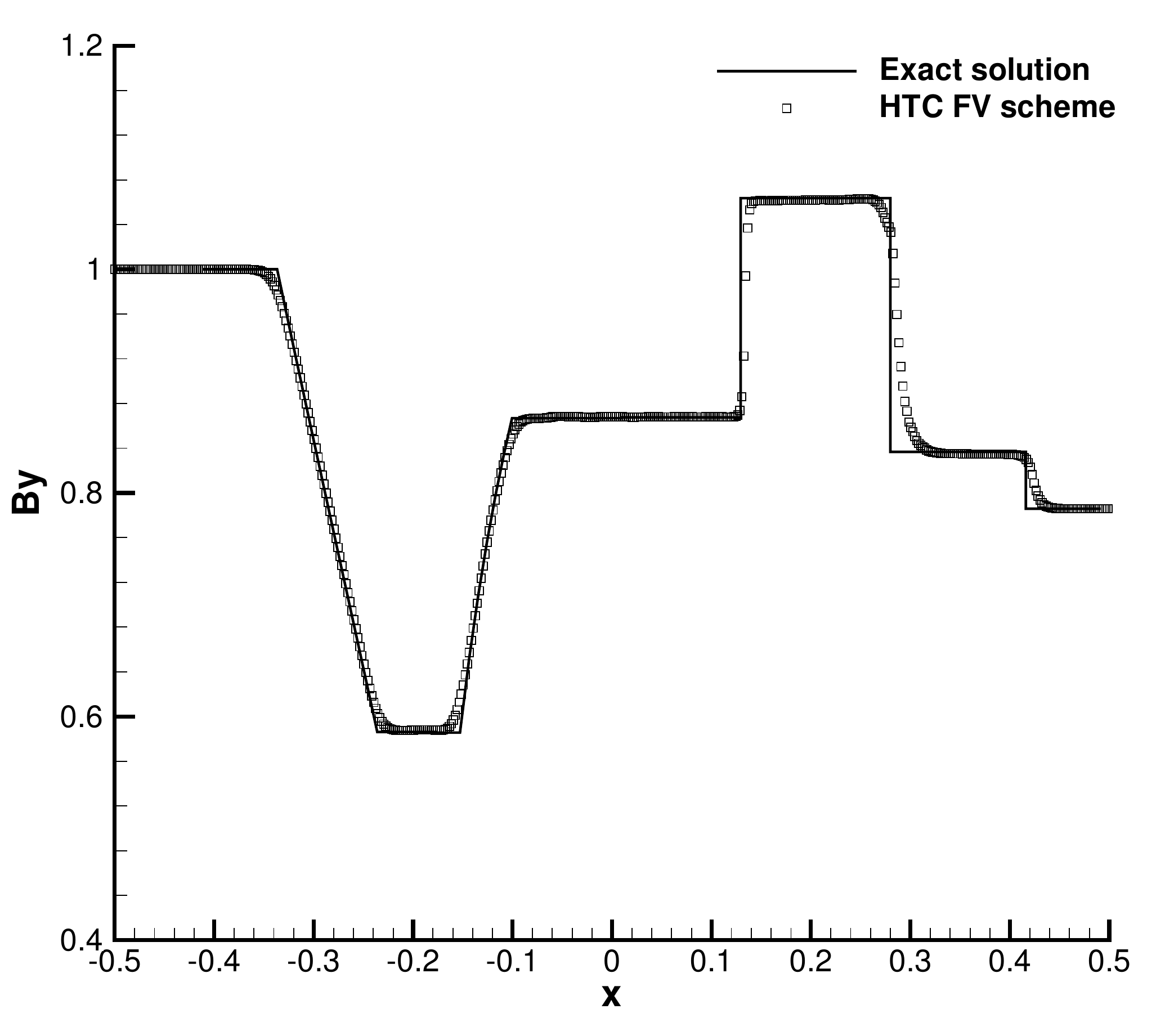}  \\    
			\includegraphics[width=0.4\textwidth]{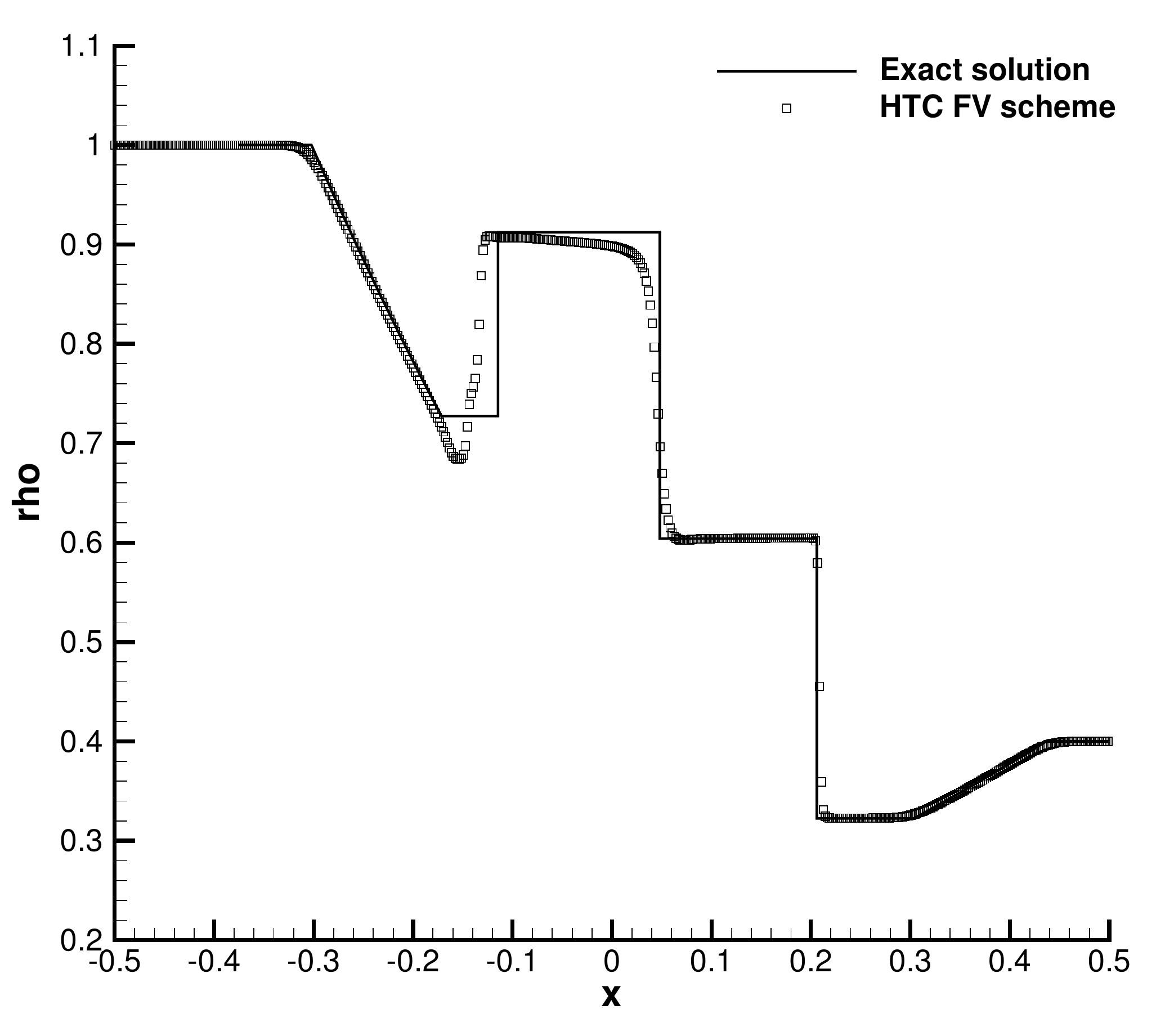} & 
			\includegraphics[width=0.4\textwidth]{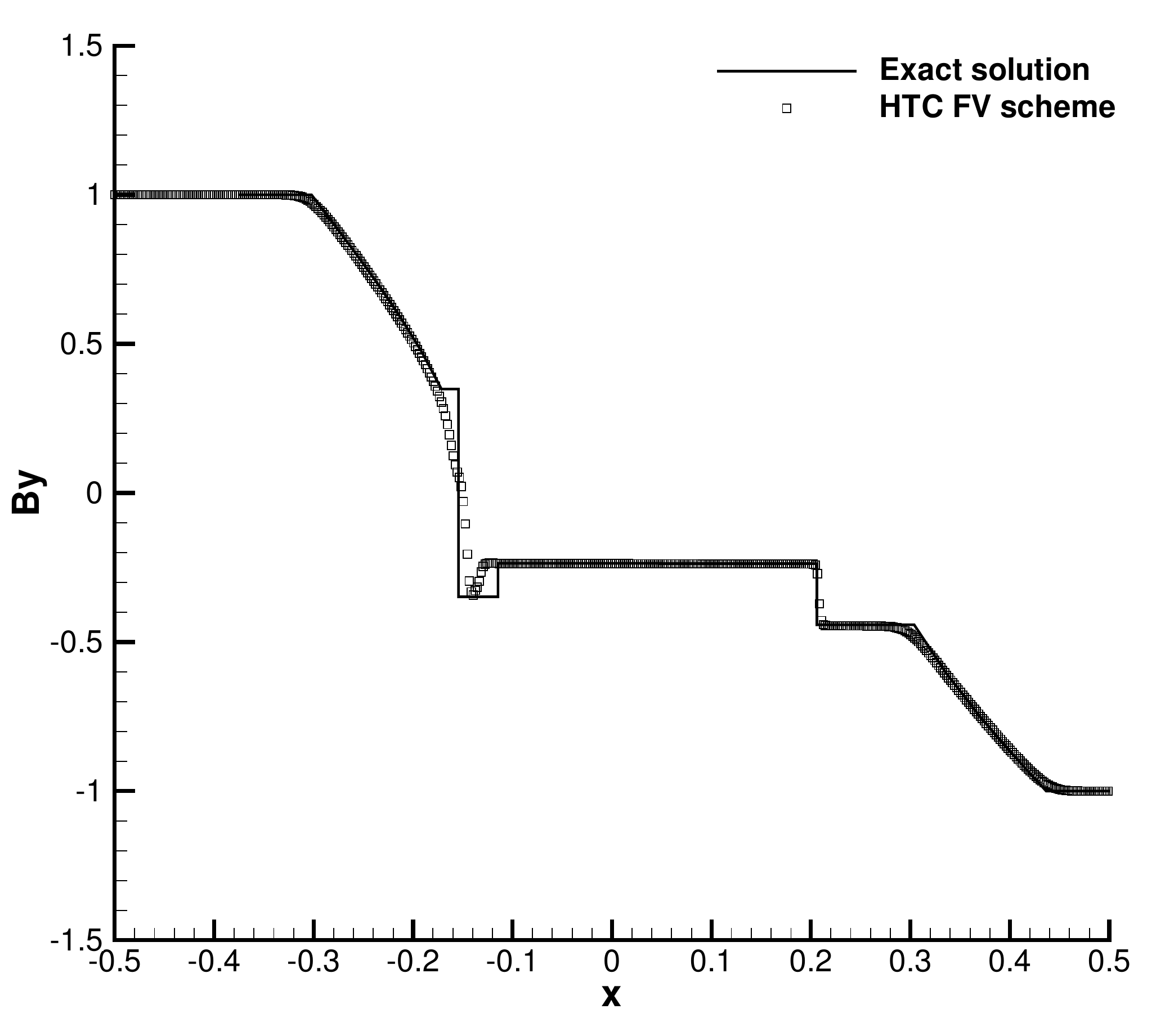}     
		\end{tabular} 
		\caption{Exact solution and numerical results obtained with the new HTC FV scheme for MHD Riemann problems RP1-RP4 (from top to bottom). Density (left) and magnetic field component $B_y$ (right) at the final times $t=0.1$, $t=0.2$, $t=0.15$ and $t=0.16$, respectively. } 
		\label{fig.rp123}
	\end{center}
\end{figure}
\subsection{Orszag-Tang vortex system} 
\label{sec.ot} 
We now study the well-known Orszag-Tang vortex system, using the computational setup provided in \cite{JiangWu,SIMHD}.  
The computational domain is $\Omega = [0,2\pi]^2$ with periodic boundary conditions everywhere. The initial conditions for the physical variables are 
$\rho = \gamma^2$, $\mathbf{v} = (-\sin(y),\sin(x),0)$, $p=\gamma$ and $\mathbf{B}= (-\sin(y),\sin(2x),0)$ with $\gamma = 5/3$ and a constant numerical viscosity of $\epsilon = 2 \cdot 10^{-3}$. 
The domain is discretized via a uniform Cartesian mesh with $1000 \times 1000$ cells. The results obtained with the HTC FV scheme are presented in Figure \ref{fig.otv} at times $t=0.5$, $t=2.0$, $t=3.0$ and $t=5.0$. They agree qualitatively well with those presented elsewhere in the literature, see e.g. 
\cite{balsarahlle2d,ADERdivB,SIMHD}. 

\begin{figure}[!htbp]
	\begin{center}
		\begin{tabular}{cc} 
			\includegraphics[trim=10 10 10 10,clip,width=0.45\textwidth]{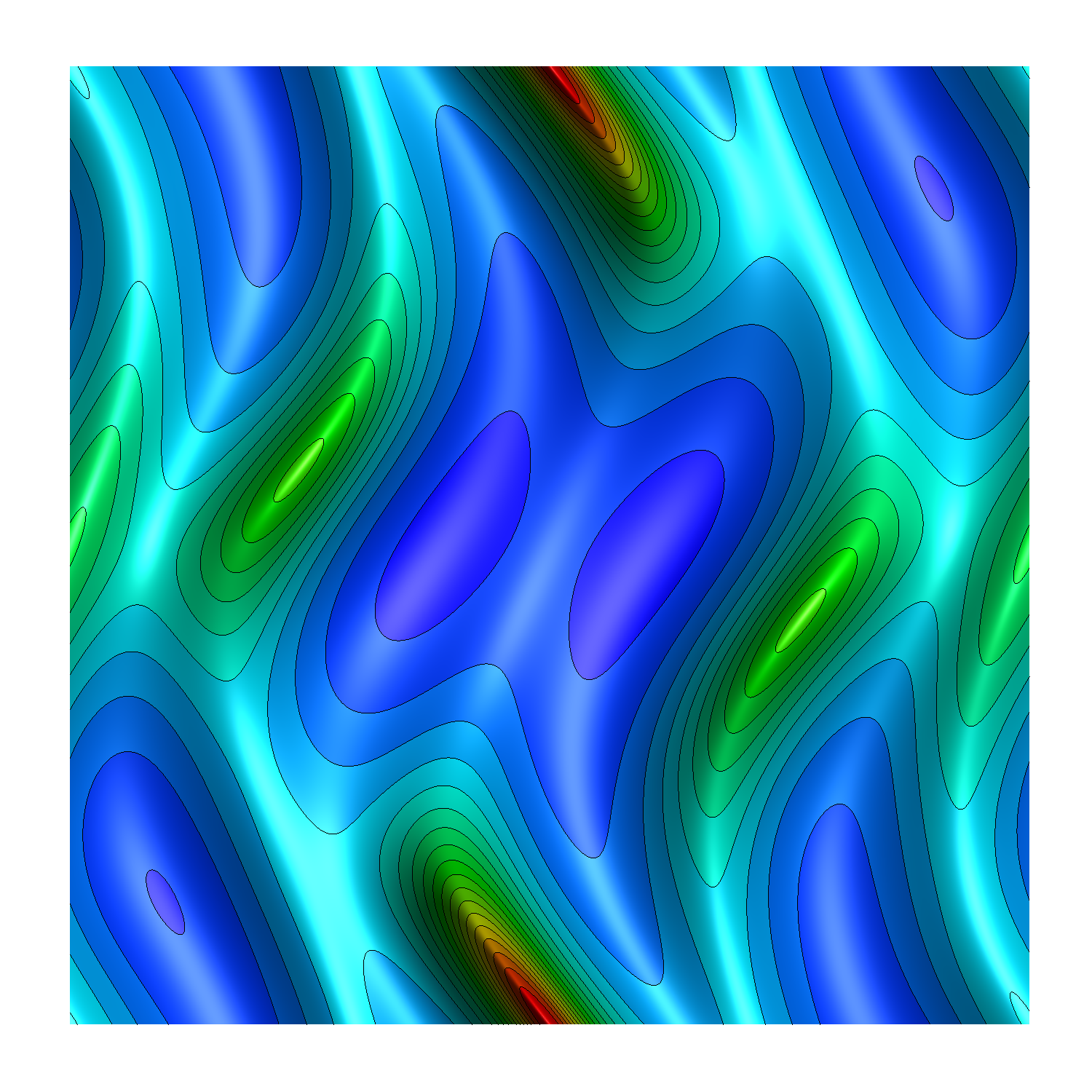} & 
			\includegraphics[trim=10 10 10 10,clip,width=0.45\textwidth]{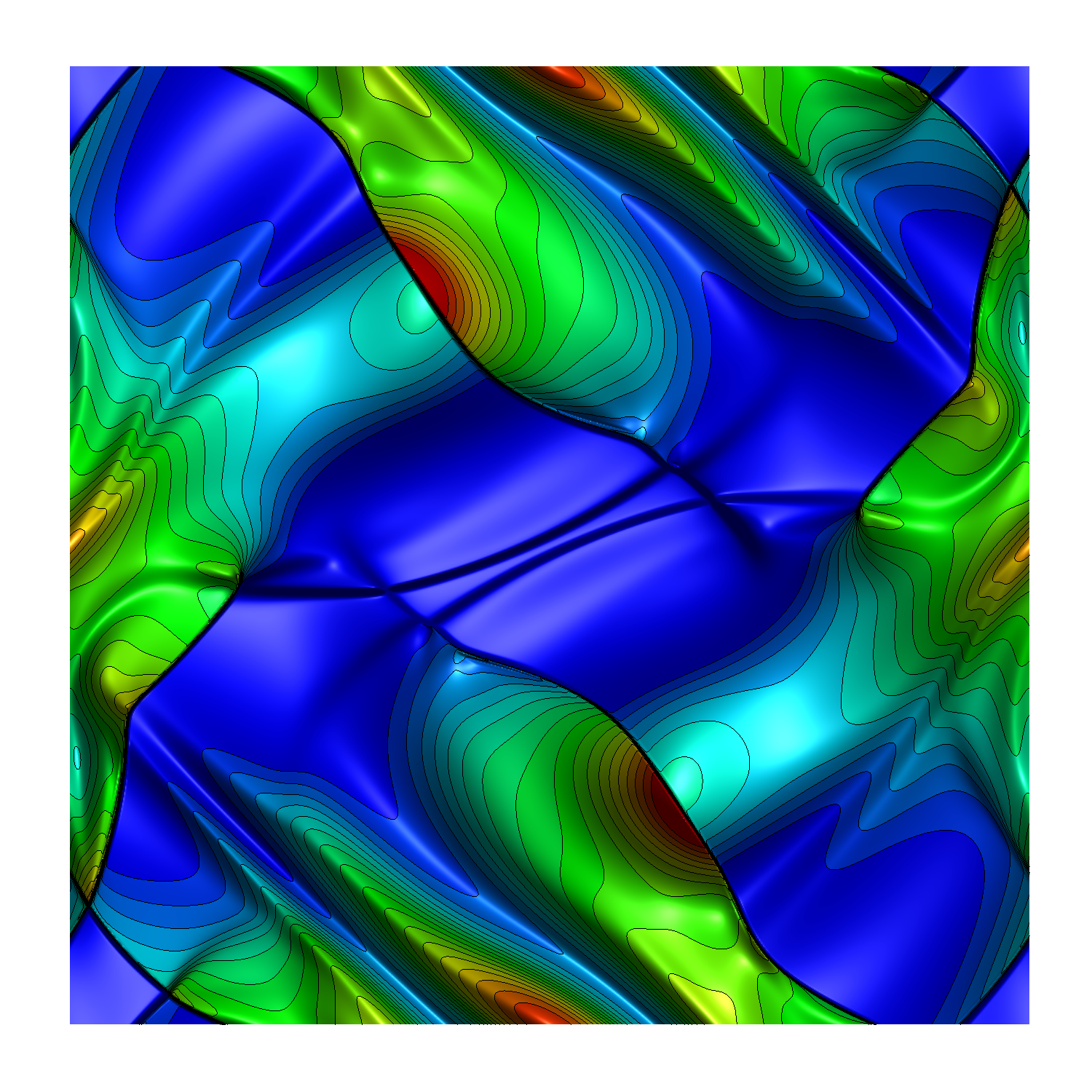}   \\ 
			\includegraphics[trim=10 10 10 10,clip,width=0.45\textwidth]{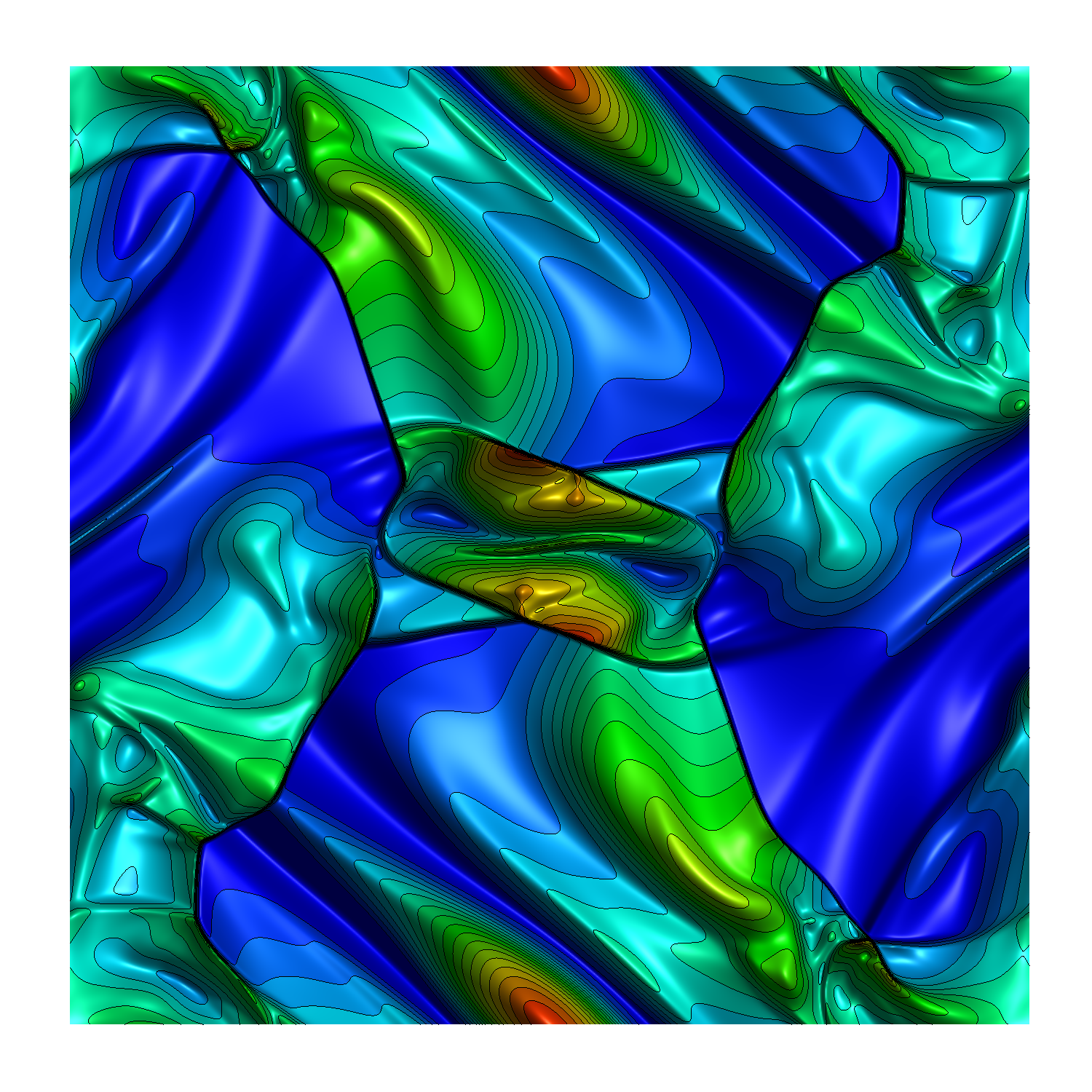} & 
			\includegraphics[trim=10 10 10 10,clip,width=0.45\textwidth]{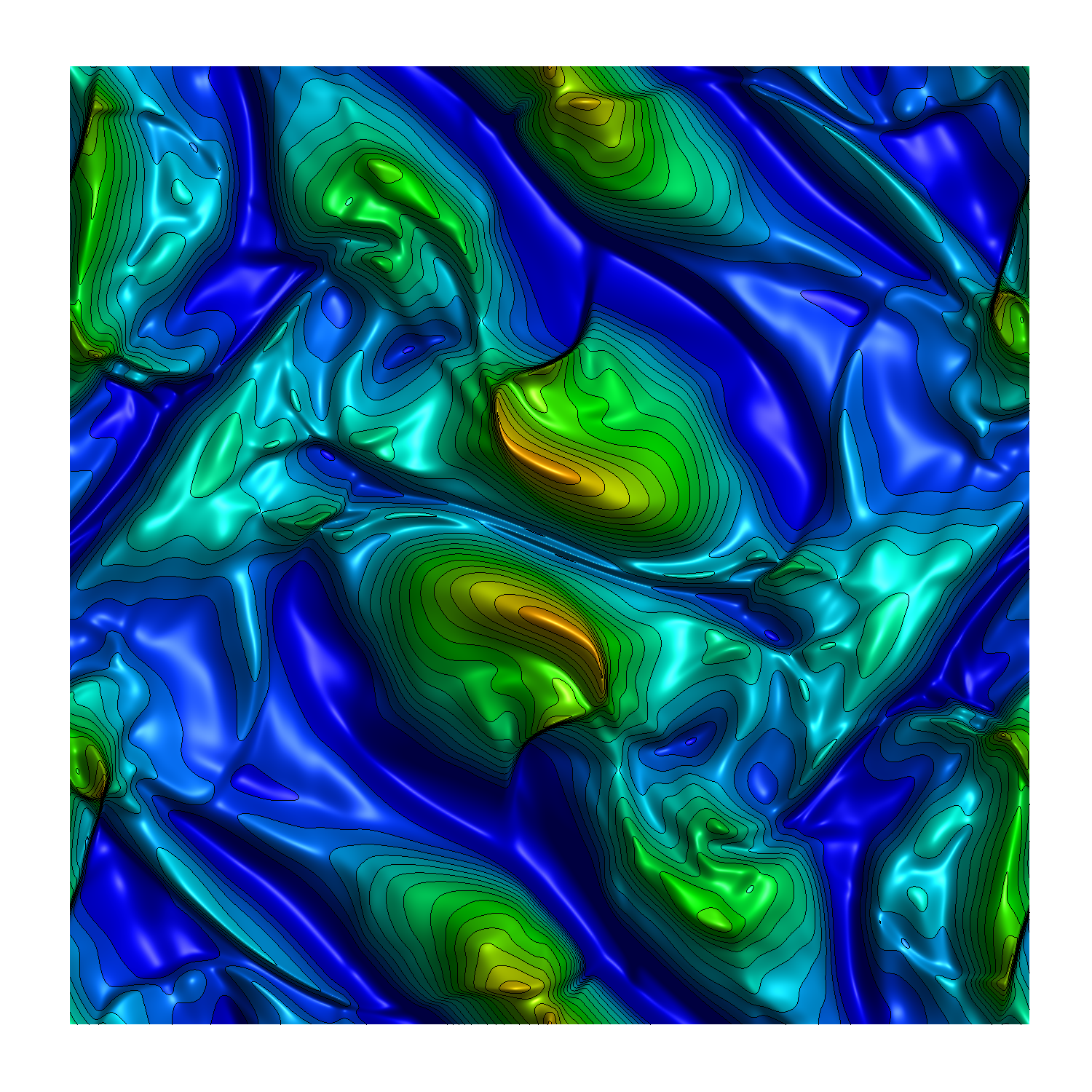}    
		\end{tabular} 
		\caption{Density contours obtained with the new semi-discrete thermodynamically comaptible finite volume scheme for the inviscid Orszag-Tang vortex system at time $t=0.5$ (top left), 
			$t=2.0$ (top right), $t=3.0$ (bottom left) and $t=5.0$ (bottom right). } 
		\label{fig.otv}
	\end{center}
\end{figure}
\subsection{MHD rotor problem} 
\label{sec.rotor} 
The MHD rotor problem is a classical MHD benchmark and was introduced for the first time in \cite{BalsaraSpicer1999}. The initial pressure and the magnetic field are set to constant values in the entire domain and are chosen as $p=1$ and $\mathbf{B} = ( 2.5 / \sqrt{4 \pi}, 0, 0)^T$.   
The computational domain $\Omega = [-0.5,0.5]^2$ is discretized at the aid of a Cartesian mesh of $1000 \times 1000$ cells. For $0 \leq \left\| \mathbf{x} \right\| \leq 0.1$, the initial density and velocity are set to $\rho=10$ and $\mathbf{v} = \boldsymbol{\omega} \times \mathbf{x}$ with  $\boldsymbol{\omega}=(0,0,10)$, respectively, while for $\left\| \mathbf{x} \right\| > 0.1$ the density and velocity are $\rho=1$ and  $\mathbf{v}=(0,0,0)$. The numerical results obtained for $\gamma=1.4$ and $\epsilon = 10^{-4}$ 
using the new HTC finite volume method are depicted in Figure \ref{fig.rotor} at time $t=0.25$ for the density, the pressure, the Mach number and the magnetic pressure. Comparison with results on the  available literature show a good qualitatively agreement, see e.g. \cite{BalsaraSpicer1999,balsarahlle2d,ADERdivB,SIMHD}. 
\begin{figure}[!htbp]
	\begin{center}
		\begin{tabular}{cc} 
			\includegraphics[trim=10 10 10 10,clip,width=0.45\textwidth]{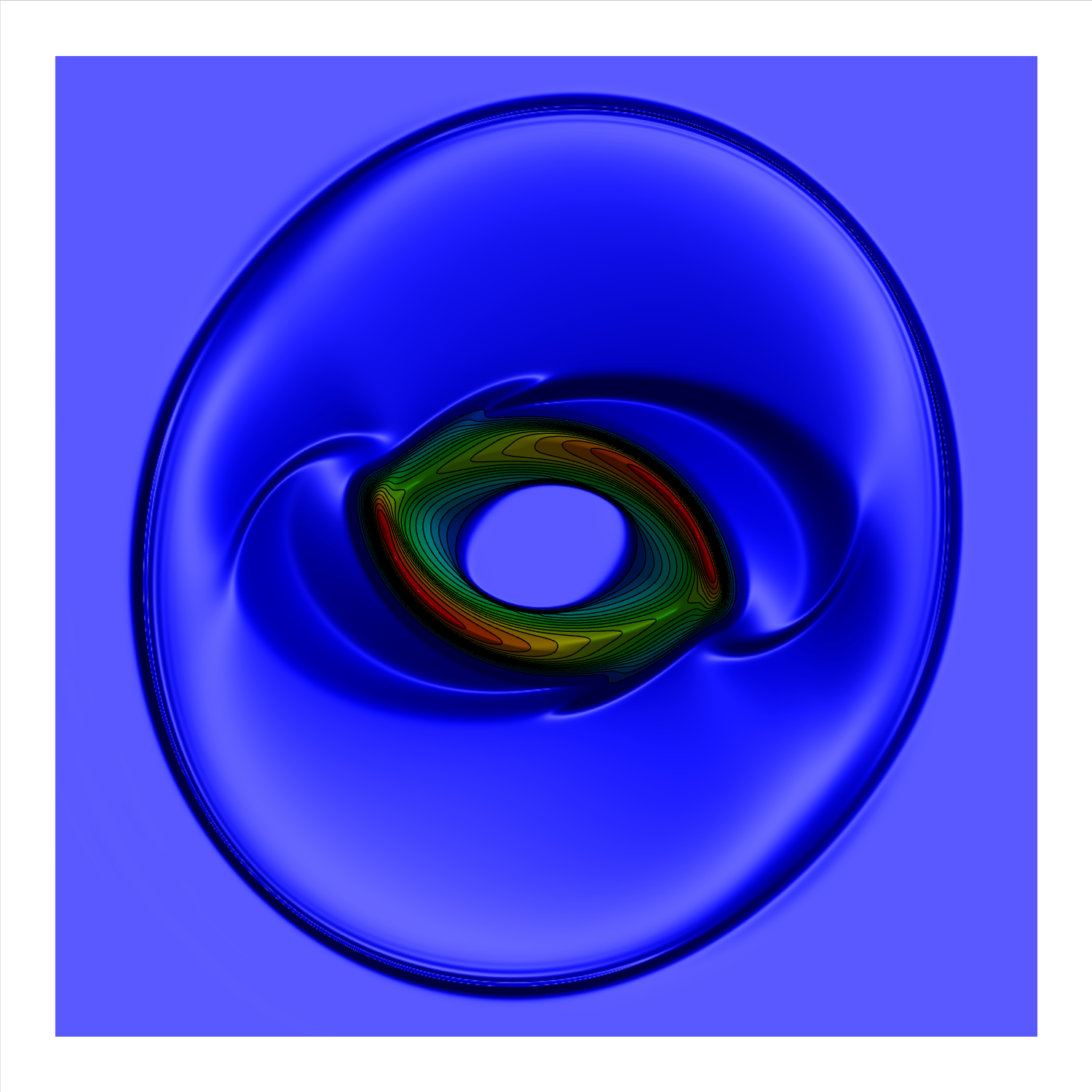} & 
			\includegraphics[trim=10 10 10 10,clip,width=0.45\textwidth]{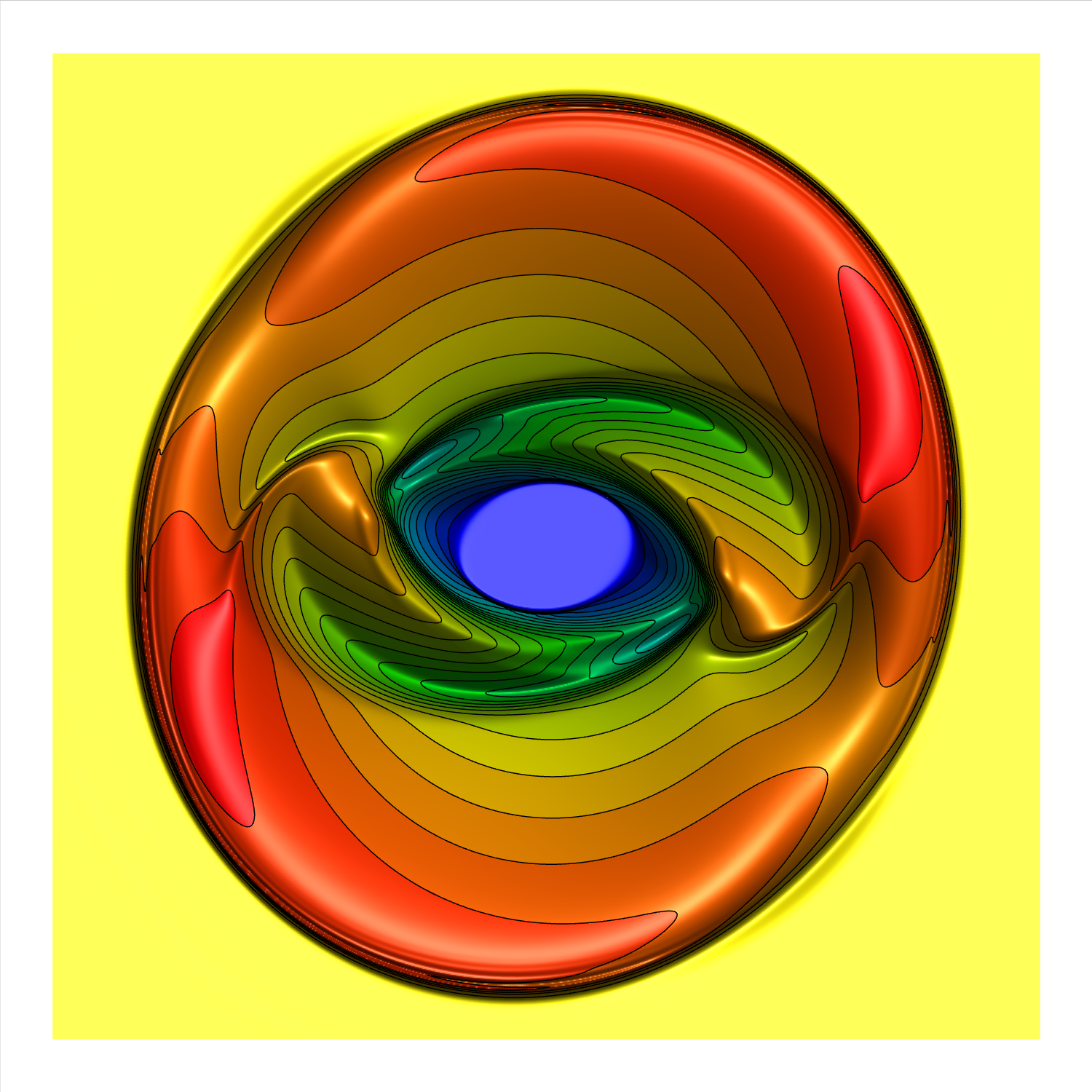}   \\ 
			\includegraphics[trim=10 10 10 10,clip,width=0.45\textwidth]{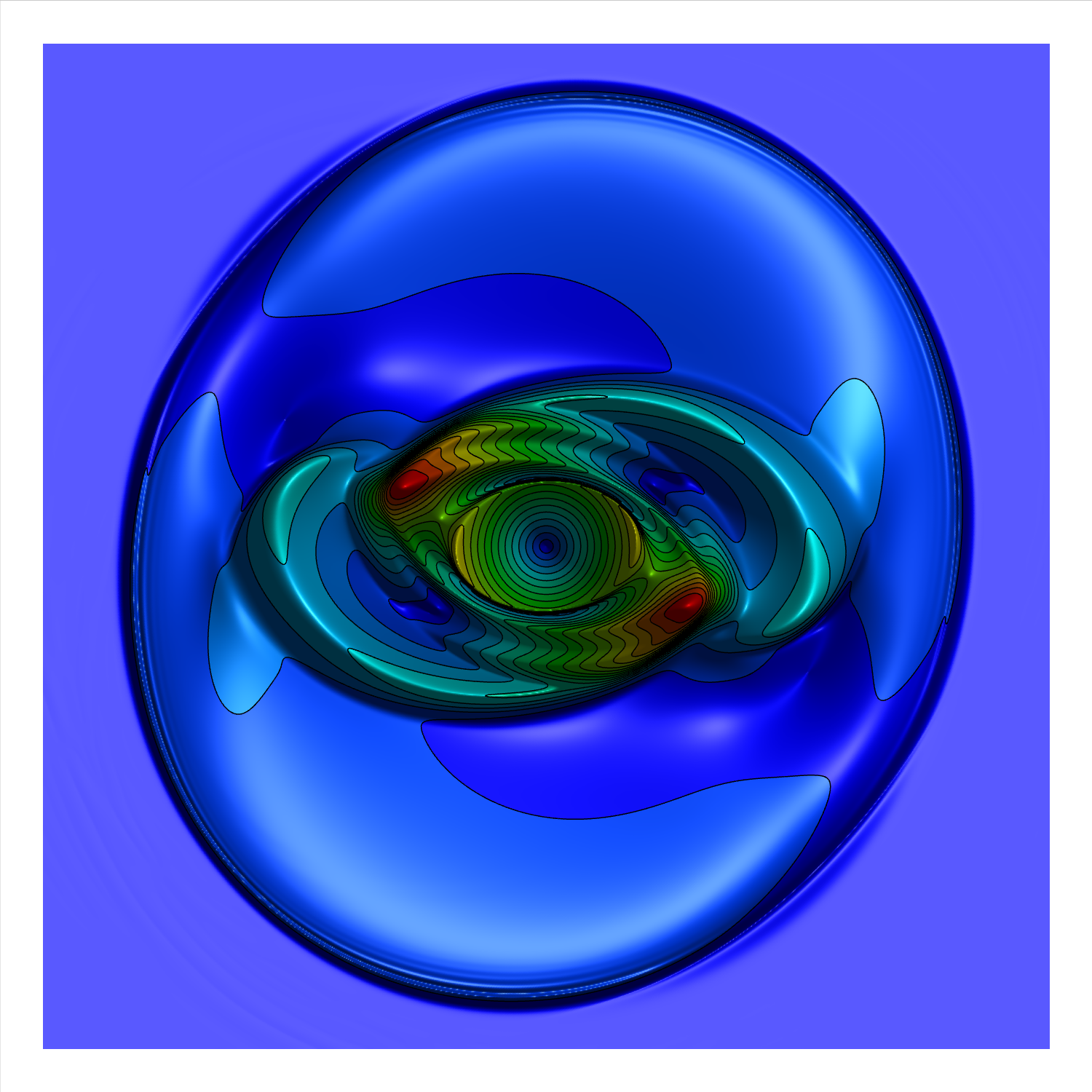} & 
			\includegraphics[trim=10 10 10 10,clip,width=0.45\textwidth]{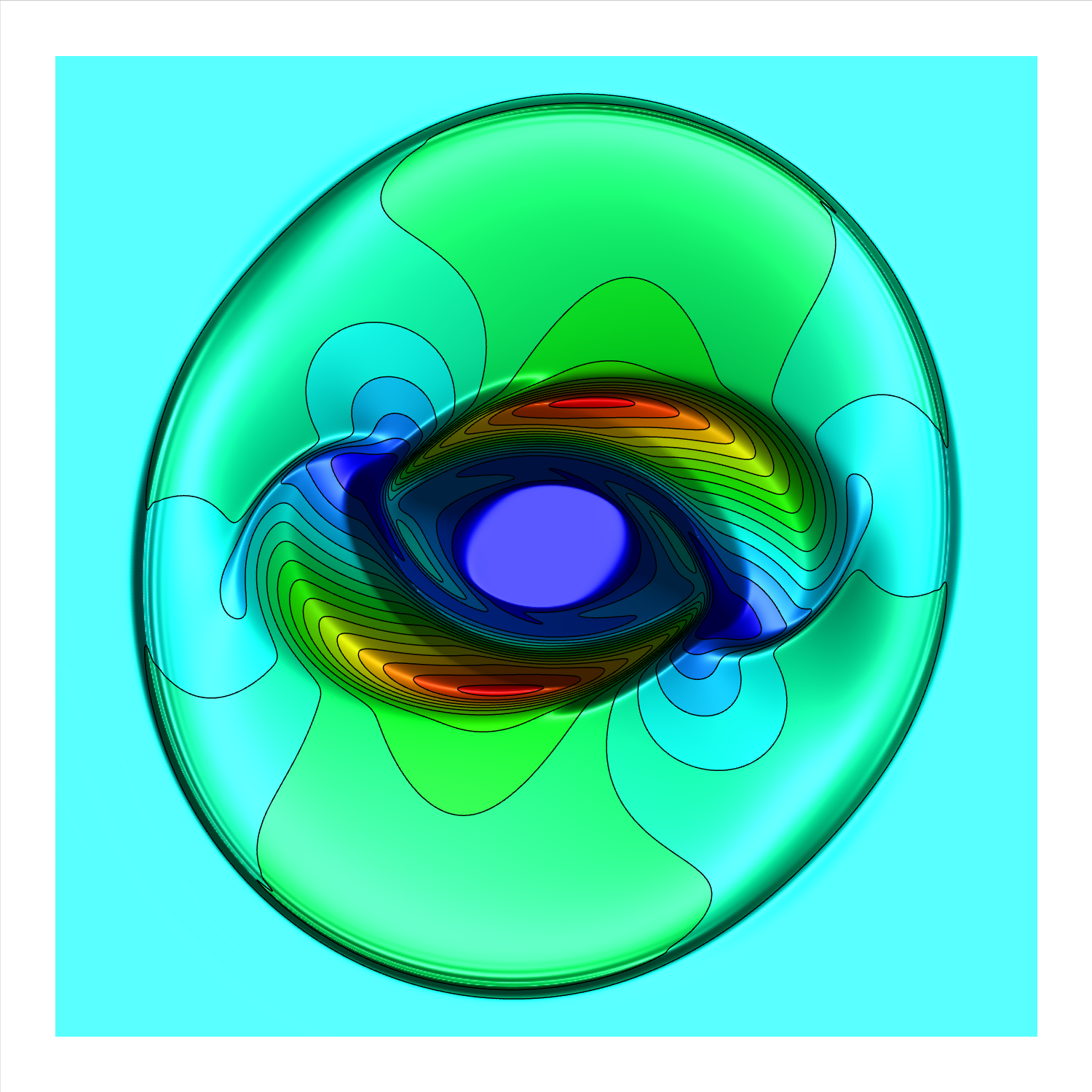}    
		\end{tabular} 
		\caption{Numerical solution obtained with the HTC finite volume method for the MHD rotor problem at time $t=0.25$. Contour lines of density (top left), pressure (top right), Mach number (bottom left) and magnetic pressure (bottom right). }  
		\label{fig.rotor}
	\end{center}
\end{figure}
\textcolor{black}{In the following we provide some numerical evidence that in our scheme the discrete entropy inequality, according to Theorem \ref{thm.entropy}, is satisfied. For this purpose in Figure \ref{fig.entropy} we plot the time series of the integral of the entropy density for the present MHD rotor problem and for the previous Orszag-Tang vortex system. As one can observe, the physical entropy is never decreasing in time, as expected. }
\begin{figure}[!htbp]
	\begin{center}
		\begin{tabular}{cc} 
			\includegraphics[width=0.45\textwidth]{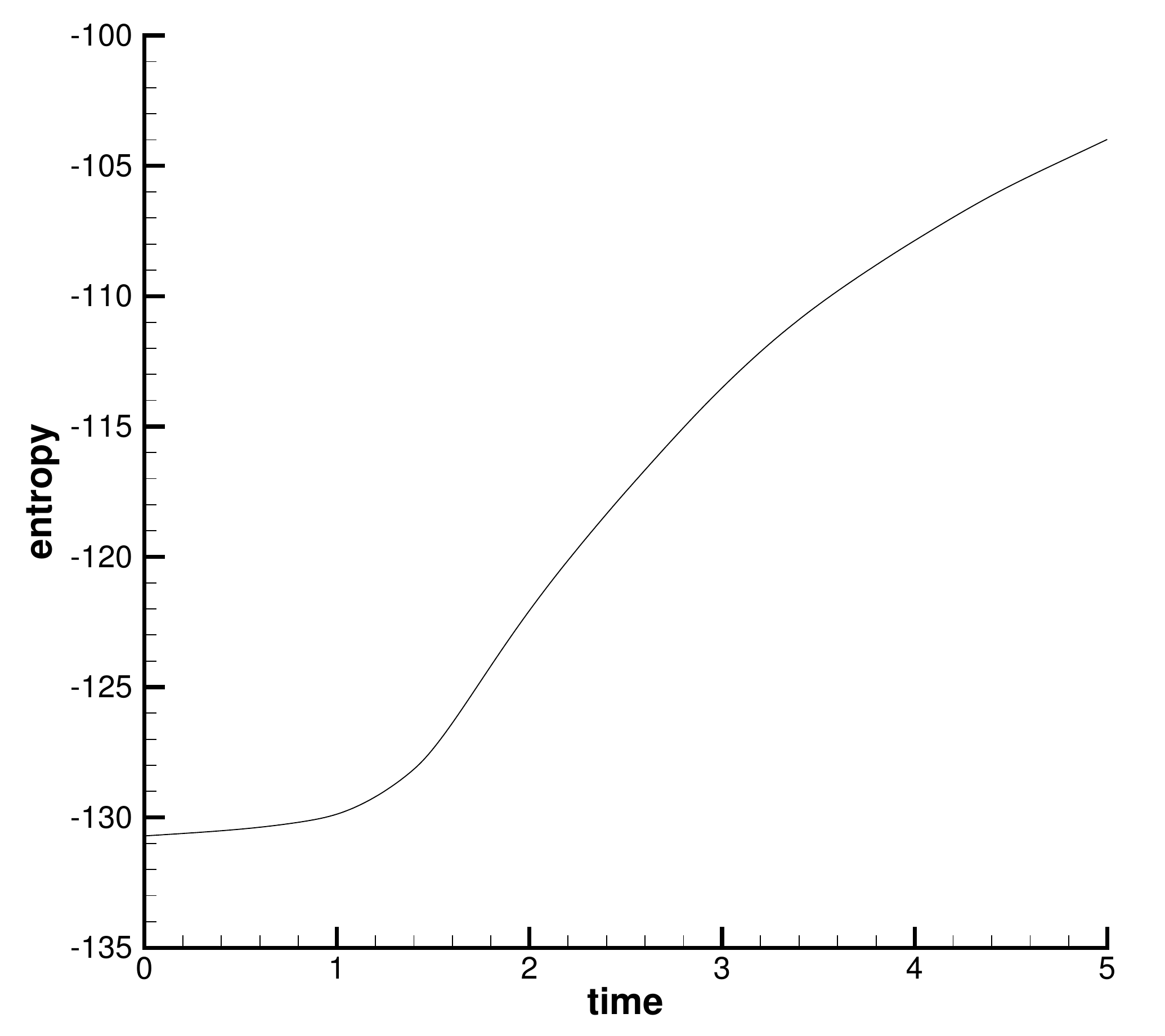} & 
			\includegraphics[width=0.45\textwidth]{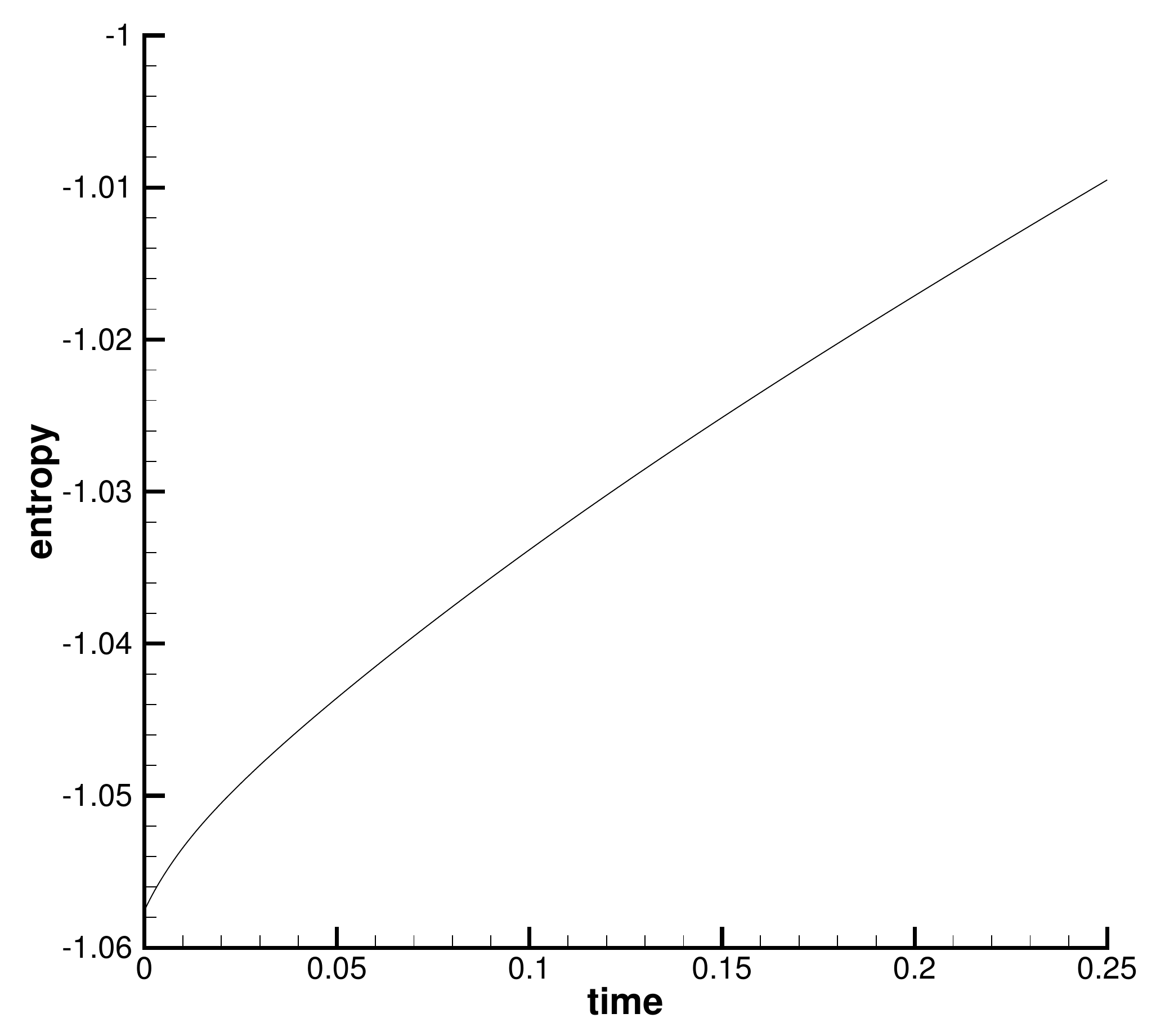}  
		\end{tabular} 
		\caption{\textcolor{black}{Time series of the integral of the physical entropy density over the computational domain. Left: Orszag-Tang vortex system. Right: MHD rotor problem. In both cases, the physical entropy is never decreasing, as expected. } }  
		\label{fig.entropy}
	\end{center}
\end{figure}

\subsection{MHD blast wave problem} 
\label{sec.blast} 
The last test case under consideration is the MHD blast wave problem introduced in \cite{BalsaraSpicer1999}, which is well known to be very challenging for numerical methods.  
Initially the density, velocity and magnetic field are set to constant values $\rho=1$, $\mathbf{v}=(0,0,0)$ and $\mathbf{B}=(100/\sqrt{4 \pi},0,0)$.   
The initial pressure jumps over four orders of magnitude and is set to $p=1000$ for $\left\| \mathbf{x} \right\| < 0.1$ and to $p=0.1$ everywhere else. We set $\gamma=1.4$ and $\epsilon=5 \cdot 10^{-3}$.     
The domain $\Omega=[-0.5,0.5]^2$ is discretized with a uniform Cartesian grid composed of $1000 \times 1000$ cells. The results obtained with the new HTC finite volume scheme are shown in Figure \ref{fig.blast} at time $t=0.01$ for the density, the pressure, the velocity magnitude and the magnetic pressure. The results agree qualitatively well with those of the literature, see \cite{BalsaraSpicer1999,ADERdivB,SIMHD}.   

\begin{figure}[!htbp]
	\begin{center}
		\begin{tabular}{cc} 
			\includegraphics[trim=10 10 10 10,clip,width=0.45\textwidth]{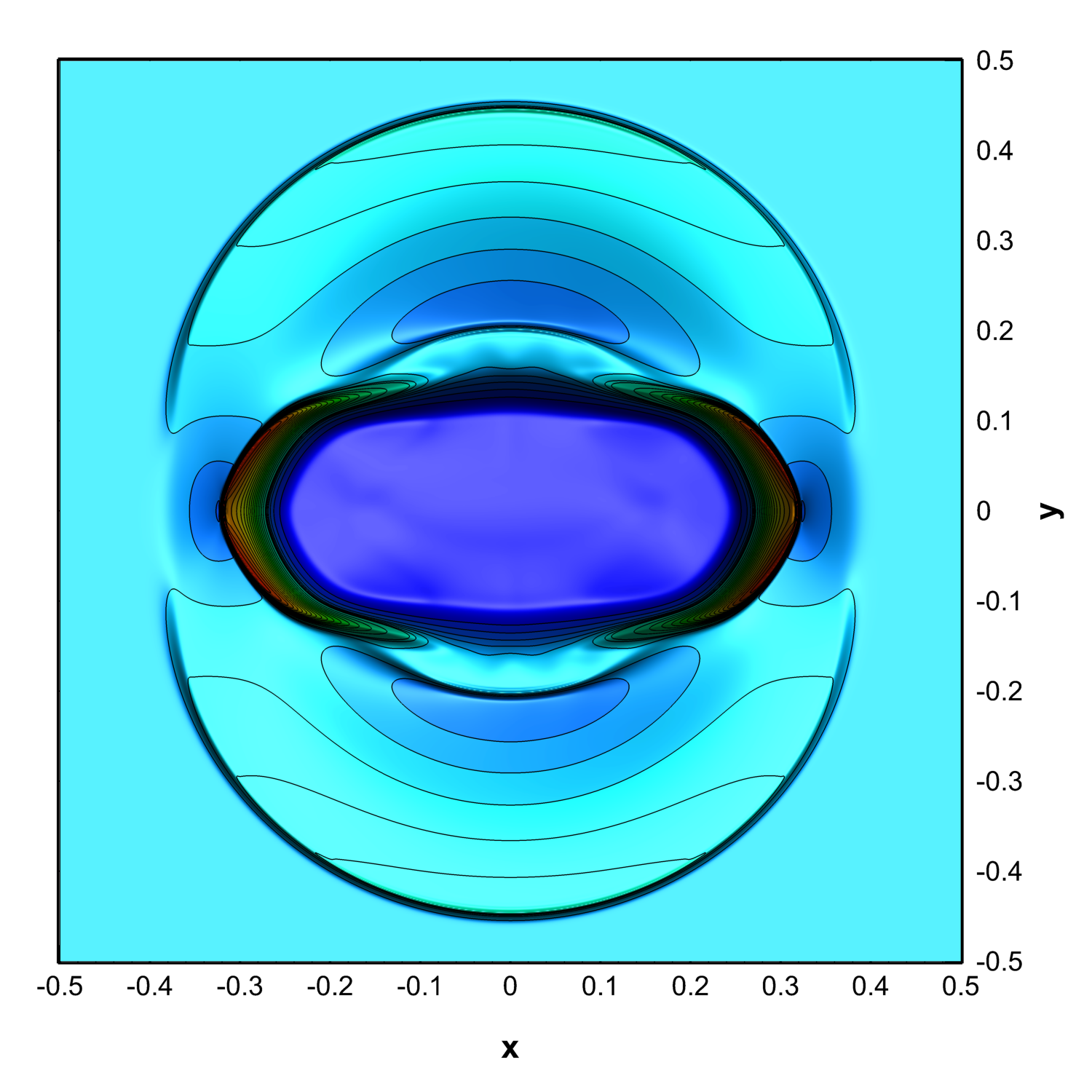} & 
			\includegraphics[trim=10 10 10 10,clip,width=0.45\textwidth]{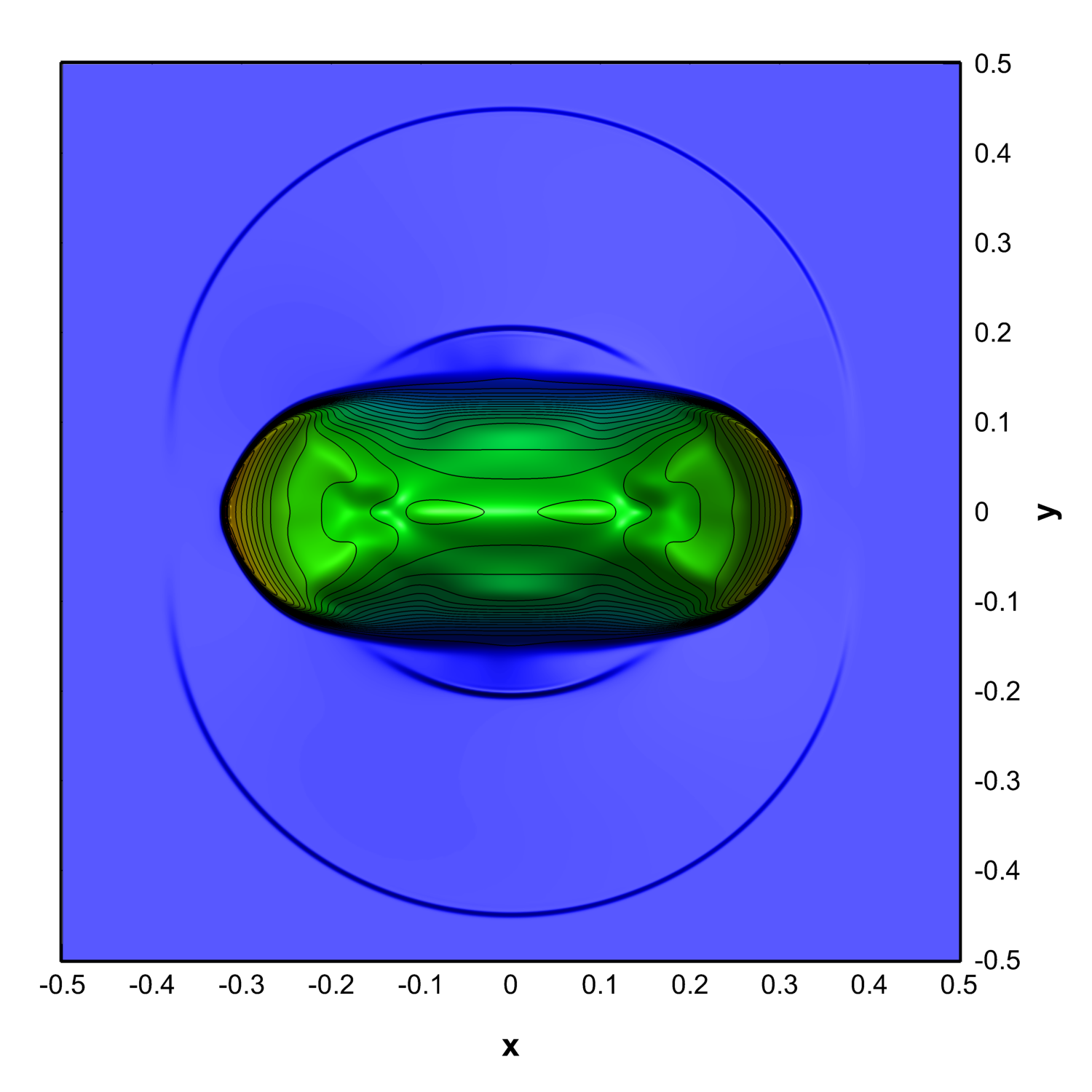}   \\ 
			\includegraphics[trim=10 10 10 10,clip,width=0.45\textwidth]{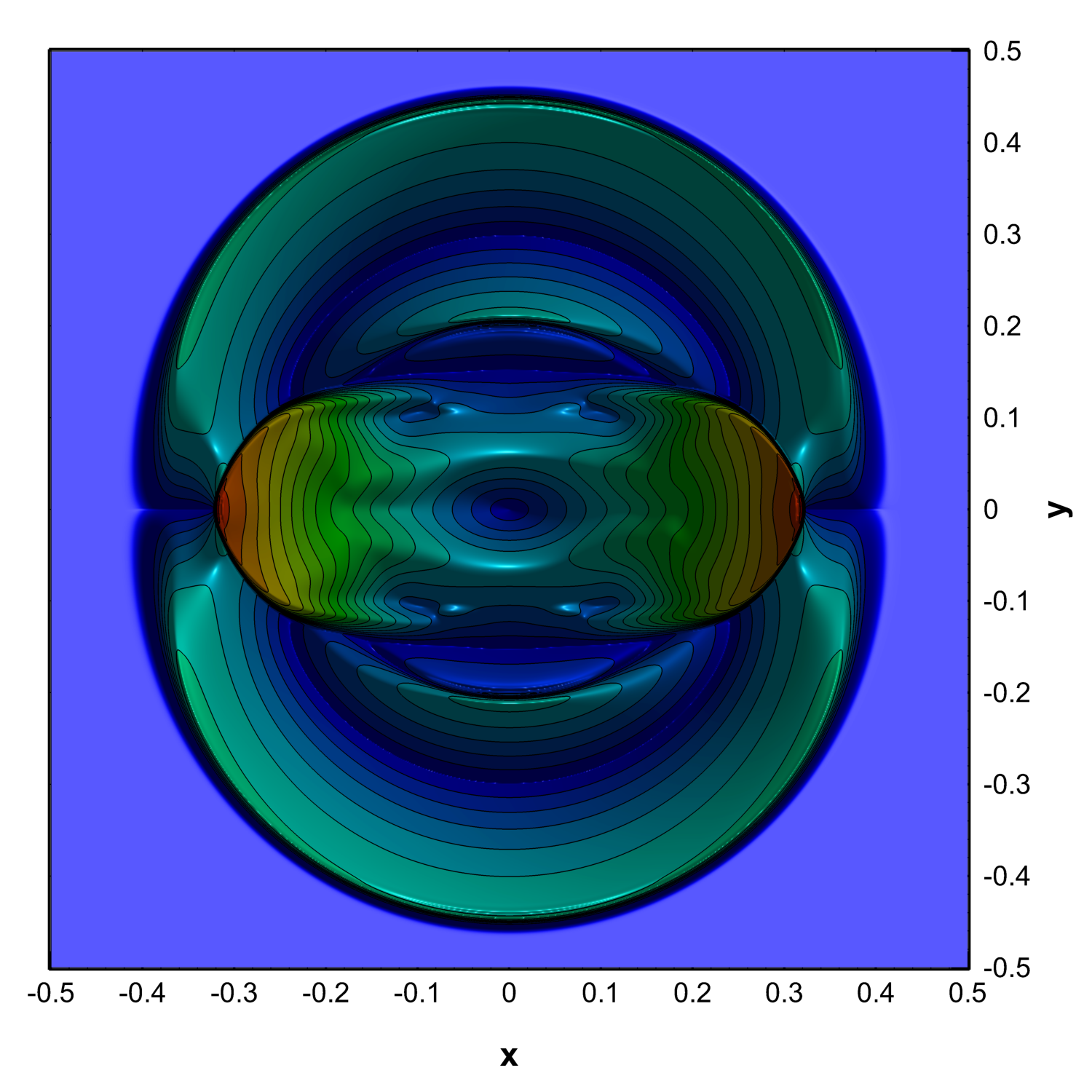} & 
			\includegraphics[trim=10 10 10 10,clip,width=0.45\textwidth]{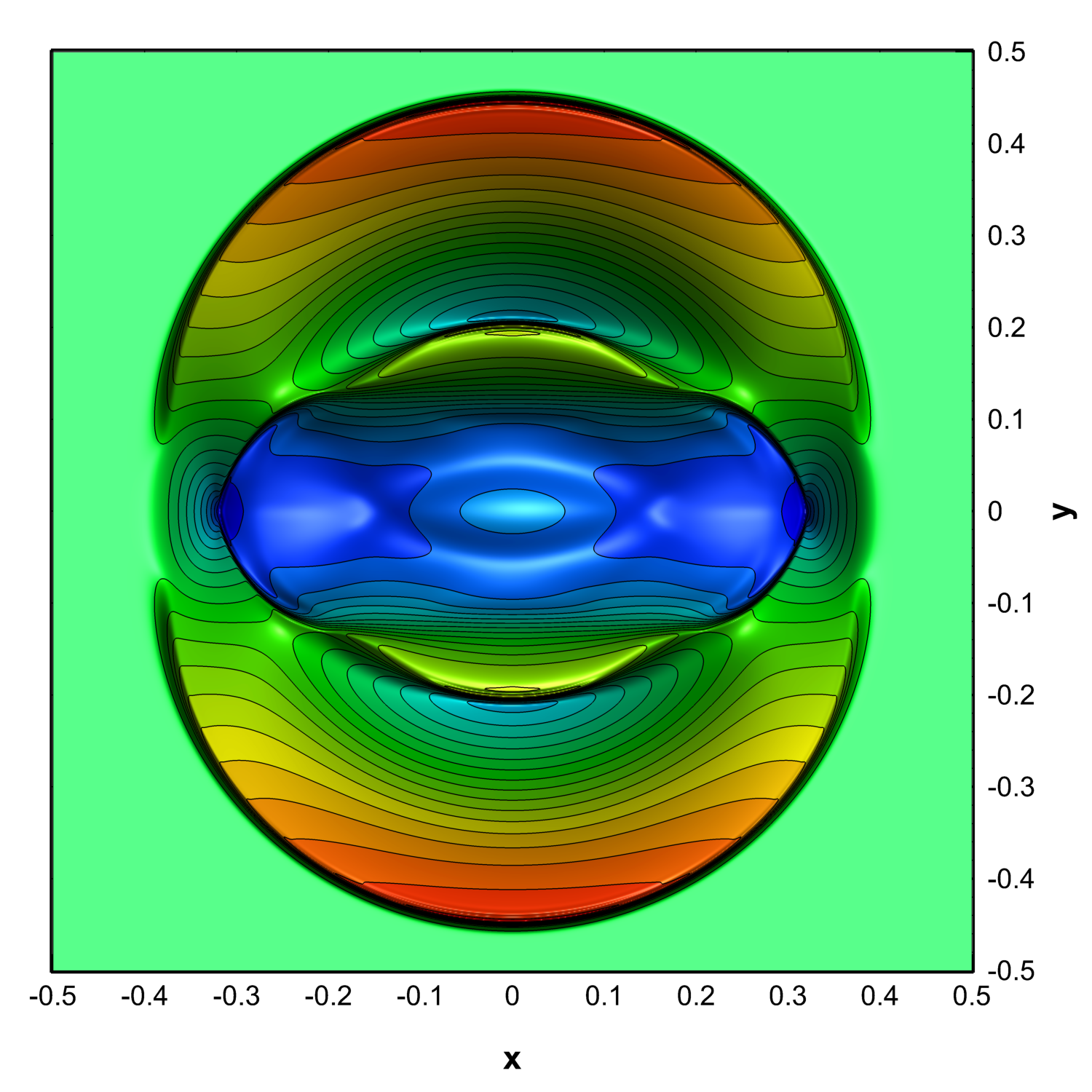}    
		\end{tabular} 
		\caption{Numerical solution obtained with the new semi-discrete thermodynamically comaptible finite volume scheme for the MHD blast wave problem at time $t=0.01$. Contour lines of density (top left), pressure (top right), velocity magnitude (bottom left) and magnetic pressure (bottom right). } 
		\label{fig.blast}
	\end{center}
\end{figure}

\section{Conclusions}
\label{sec.Conclusions}

In this paper, we have presented a new thermodynamically compatible semi-discrete finite volume scheme (HTC scheme) for the equations of ideal magnetohydrodynamics (MHD). Unlike classical finite volume schemes for MHD the new method directly discretizes the \textit{entropy inequality} and \textit{not} the total energy conservation law. The discrete total energy conservation is instead achieved as a mere \textit{consequence} of a suitable and thermodynamically compatible discretization of all the other equations. The divergence-free constraint of the magnetic field was taken into account at the aid of a hyperbolic and thermodynamically compatible GLM divergence cleaning, following the seminal ideas of Munz \textit{et al.} \cite{MunzCleaning,Dedneretal} on hyperbolic GLM techniques for the preservation of involution constraints in hyperbolic PDE systems. 
The finite volume scheme proposed in this paper satisfies a discrete entropy inequality \textit{by construction} and can be proven to be \textit{nonlinearly stable} in the energy norm, as total energy is  conserved \textcolor{black}{up to errors due to numerical quadrature and the time discretization, see \cite{HTCGPR} for details on the influence of the numerical quadrature and the time discretization on energy conservation errors}. 
The new method has been shown to be second order accurate and has been applied to some classical MHD benchmark problems in one and two space dimensions, obtaining a good agreement with existing reference solutions available in the literature. 

Future work will concern the development of suitable \textit{symplectic} time integrators, in order to conserve the discrete total energy also exactly on the fully discrete level, see e.g. \cite{Brugnano1,Brugnano2}. 
Further research is needed to extend the present scheme to higher order in space within the discontinuous Galerkin (DG) finite element framework, similar to the entropy compatible DG schemes introduced in \cite{GassnerEntropyGLM,ShuEntropyMHD2,GassnerSWE}.   
Another major challenge that is left to future work is the development of thermodynamically compatible finite volume methods that also preserve the divergence constraint of the magnetic field \textit{exactly} at the semi-discrete level. For that purpose, we will consider face-based / edge-based \textit{staggered} meshes as in 
 \cite{SIMHD} and \cite{Hybrid2}. Last but not least, we plan to develop new thermodynamically compatible FV schemes for the MHD equations using the general framework introduced by Abgrall in \cite{Abgrall2018}.

\section*{Acknowledgments}

S.B. and M.D. are members of the INdAM GNCS group and acknowledge the financial support received from  
the Italian Ministry of Education, University and Research (MIUR) in the frame of the PRIN 2017 project \textit{Innovative numerical methods for evolutionary partial differential equations and  applications}
and from the Spanish Ministry of Science and Innovation, grant number PID2021-122625OB-I00. 
S.B. was also funded by INdAM via a GNCS grant for young researchers and by an \textit{UniTN starting grant} of the University of Trento. 
\textcolor{black}{The authors are very grateful to the two anonymous referees for their constructive and insightful comments, which helped to improve the clarity and quality of this paper. }

\bibliographystyle{plain}
\bibliography{biblio}
\end{document}